\documentclass[envcountsect,envcountsame]{svmult}
\usepackage{amssymb}

\usepackage{makeidx}         
\usepackage{graphicx}        
\usepackage{multicol}        

\makeindex             

\usepackage{amssymb, stmaryrd, eucal, amsmath}
\usepackage{amscd}

\smartqed

\def\Max{{\rm Max}}

\def\t.d.{{\rm t.d.}}

\def\Q{{\mathbb{Q}}}

\usepackage{color}  
\definecolor{darkgreen}{rgb}{0.03, 0.5, 0.03}
 \def\D{\boldsymbol{\mathcal{D}}}
 \def\Q{\boldsymbol{\mathcal{Q}}}
 \def\Max{\mbox{\rm Max}}
 \def\Spec{\mbox{\rm Spec}}
 \def\F{\boldsymbol{F}}

\newcommand{\G}{\boldsymbol{\mathcal{G}}}
\newcommand{\K}{\boldsymbol{\mathcal{K}}}
\newcommand{\KK}{\boldsymbol{\mathcal{H}}}
\newcommand{\U}{\boldsymbol{\mathcal{U}}}

\newcommand{\Cl} {\mbox{\rm Cl}}

\newcommand{\co} {\boldsymbol c}
\newcommand{\Na} {\mbox{\rm Na}}
 \newcommand{\Kr} {\mbox{\rm Kr}}

\newcommand{\cl} {\mbox{\texttt cl}}

\newcommand{\f}{\boldsymbol{{f}}}

\newcommand{\astf} {\ast{_{\!{_f}}}}

\newcommand{\tV}{{\mbox{\texttt{V}}}}
\newcommand{\tT}{{\mbox{\texttt{T}}}}
\newcommand{\tN}{{{\texttt{N}}}}
 \newcommand{\tS}{{\mbox{\texttt{S}}}}

\begin{document}

\title*{On $v$-domains: a survey}
\author{Marco Fontana \and Muhammad Zafrullah}
\institute{Marco Fontana 
 \at Dipartimento di Matematica,
Universit\`a degli Studi ``Roma Tre'',
00146 Roma, Italy,\\ \email{\texttt{fontana@mat.uniroma3.it}}\\ \\
Muhammad Zafrullah \at 57 Colgate Street, Pocatello, ID 83201, USA, \\\email{\texttt{mzafrullah@usa.net}}}

%
%

\maketitle


\abstract{
An integral domain $D$ is a $v$--domain if, for every finitely generated
nonzero  (fractional)   ideal $F$ of $D$, we have $(FF^{-1})^{-1}=D$. The $v$--domains
generalize Pr\"{u}fer   and Krull  domains and have appeared in the literature with
different names. This  paper   is the result of an effort to put together
information on this useful class of integral domains. In this survey, we 
 present  old, recent and new characterizations of $v$--domains along
with some historical remarks. We  also discuss   the relationship of $v$--domains with their various specializations and generalizations, giving
suitable examples. 
 }

\begin{section}{Preliminaries and Introduction} \label{prel}

 Let $D$ be an integral domain with quotient field $K$. Let $
\boldsymbol{\overline{F}}(D)$ be the set of all nonzero
$D$--submodules of $K$ and let $\boldsymbol{F}(D)$ be the set of
all nonzero fractional ideals of $D$, i.e., $A \in
\boldsymbol{F}(D)$ if $A \in \boldsymbol{ \overline{F}}(D)$ and
there exists an element $0 \ne d \in D$ with $dA \subseteq D$. Let
$\boldsymbol{f}(D)$ be the set of all nonzero finitely generated
$D$--submodules of $K$. Then, obviously $\boldsymbol{f}(D)
\subseteq \boldsymbol{F}(D) \subseteq
\boldsymbol{\overline{F}}(D)$.

Recall that  a \emph{star operation} on $D$ is a map $\ast: \boldsymbol{{F}}
(D) \to \boldsymbol{{F}}(D),\  A \mapsto A^\ast$, such that the
following properties hold for all $0 \ne x \in K$ and all $A,B \in
\boldsymbol{{F}}(D)$:

\begin{itemize}
\item[$(\ast_1)$] \hskip 20pt $D = D^\ast$, \  $(xA)^\ast=xA^\ast$;

\item[$(\ast_2)$] \hskip 20pt  $A \subseteq B$ implies $A^\ast \subseteq B^\ast$;

\item[$(\ast_3)$] \hskip 20pt $A \subseteq A^\ast$ and $A^{\ast \ast}   :=  \left(A^\ast
\right)^\ast=A^\ast$.
\end{itemize}
 (the reader may consult    \cite[Gilmer (1972), Sections
32 and 34]{Gi}   for a quick review of star
operations).

In \cite[Okabe-Matsuda (1994)]{OM2}, the authors   introduced  a useful generalization of the 
notion
of a star operation: a \emph{semistar operation} on $D$ is a map $\star: \boldsymbol{\overline{F}}
(D) \to \boldsymbol{\overline{F}}(D),\  E \mapsto E^\star$, such that the
following properties hold for all $0 \ne x \in K$ and all $E,F \in
\boldsymbol{\overline{F}}(D)$:

\begin{itemize}
\item[$(\star_1)$] \hskip 20pt $(xE)^\star=xE^\star$;

\item[$(\star_2)$] \hskip 20pt  $E \subseteq F$ implies $E^\star \subseteq F^\star$;

\item[$(\star_3)$] \hskip 20pt  $E \subseteq E^\star$ and $E^{\star \star}   :=  \left(E^\star
\right)^\star=E^\star$.
\end{itemize}

Clearly, a semistar operation  $\star$ on $D$, restricted to $\boldsymbol{{F}}(D)$,  determines a star operation if and only if $D = D^\star$.




{ If $\ast$ is a star operation on $D$, then we can consider the map\
$\ast_{\!_f}: \boldsymbol{{F}}(D) \to \boldsymbol{
{F}}(D)$ defined as follows:
\begin{equation*}
A^{\ast_{\!_f}}:=\bigcup \{F^\ast\mid \ F \in \boldsymbol{f}(D)
\mbox{ and } F \subseteq A\} \;\; \;\; \mbox{  for all } A \in
\boldsymbol{{F}}(D)  .
\end{equation*}
\noindent It is easy to see that $\ast_{\!_f}$ is a star
operation on $ D $, called the \emph{star operation of finite
type associated to $\ast$}. Note that $F^\ast=F^{\ast_{\!_f}}$
for all $F \in \boldsymbol{f}(D)$. A  star operation $\ast$ is
called a \emph{star operation of finite type} (or a
\emph{star operation of finite character}) if $
\ast=\ast_{\!_f}$. It is easy to see that
$(\ast_{\!_f}\!)_{\!_f}=\ast_{ \!_f}$ (i.e., $\ast_{\!_f}$ is
of finite type). }

If $\ast_1$ and $\ast_2$ are two  star operations on $D$, we
say that $ \ast_1 \leq \ast_2$ \ if $A^{\ast_1} \subseteq
A^{\ast_2}$ for all $A \in \boldsymbol{{F}}(D)$. This is
equivalent to saying that $ \left(A^{\ast_{1}}\right)^{\ast_{2}}
 =  A^{\ast_2}  =
\left(A^{\ast_{2}}\right)^{\ast_{1}}$ for all $A \in
\boldsymbol{{F}}(D)$. Obviously, for  any   star
operation $\ast$ on $D$, we have $\ast_{\!_f} \leq \ast$, and
if $\ast_1 \leq \ast_2$, then $ (\ast_1)_{{\!_f}} \leq
(\ast_2)_{{\!_f}}$.

Let $I\subseteq D$ be a nonzero ideal of $D$. We say that $I$ is a
 \it $\ast$--ideal \rm of $D$ if
 $I^{\ast}=I $. We call a
 $\ast $--ideal of $D$ a
\it $\ast $--prime  
 ideal  \rm  of $D$ if it is also a prime ideal and we call a
maximal element in the set of all proper  $\ast$--ideals
of $D$ a \it  $\ast$--maximal  ideal  \rm
of $D$.

It is not hard to prove that a  $\ast $--maximal ideal is a
prime ideal and that each proper  $\ast _{_{\!f}}$--ideal
is contained in a  $\ast _{_{\!f}}$--maximal ideal. 

Let $\Delta $ be a set of prime ideals of an integral domain $D$ and set \begin{equation*}
E^{\star
_{\!\Delta }}:=\bigcap \left\{ ED_{Q}\mid Q\in\Delta\right\}\;\;\;\;  \mbox{\rm  for all $E \in
\boldsymbol{\overline{F}}(D)$}.
\end{equation*}
 The  operation $\star_{\!\Delta }$ is a semistar operation on $D$
called the  \it  spectral semistar operation associated to
$\Delta $.  \rm Clearly, it gives rise to a star operation on $D$  if (and only if)  $\bigcap \left\{D_{Q}\mid Q\in\Delta\right\} = D$. \rm  

Given a star operation $\ast$ on $D$, when $\Delta $ coincides with $\text{Max}^{\ast _{_{\!f}}}(D)$,
the (nonempty) set of all  ${\ast _{_{\!f}}}$--maximal ideals of $D$,  the operation $\widetilde{\ast }$ defined as follows:
\begin{equation*}
A^{\widetilde{\ast }}:=\bigcap \left\{ AD_{Q}\mid Q\in \text{Max}^{\ast _{_{\!f}}}(D)\right\}\;\;\;\;  \mbox{\rm  for all $A \in
\boldsymbol{{F}}(D)$}
\end{equation*} 
determines a star operation on $D$, called the
\textit{stable star operation of finite type associated to
$\ast$}. It is not difficult to show that $
\widetilde{\ast }\leq \ast _{_{\!f}}\leq \ast $.

It is easy to see that, {\sl mutatis mutandis}, all the previous notions can be extended to the case of a semistar operation.

\medskip

Let $A, B\in \boldsymbol{{F}}(D)$, set $(A:B) :=\{z\in K\mid zB\subseteq A\}$, $(A:_D B):= (A:B) \cap D$, $A^{-1}  :=
 (D:A)$. 
 As usual, we let  $ v_{D}$ (or just
$v$) denote the \emph{$v$--operation} defined by
$A^{v}:=(D:(D:A))=\left( A^{-1}\right) ^{-1}$ for all $A\in
\boldsymbol{{F}}(D)$.  Moreover, we denote $(v_{D})_{_{\!f}}$ by $t_{D}$
(or just by $t$), the \emph{$t$--operation} on $D$; and we
denote the stable semistar operation of finite type associated to
$v_{D}$ (or, equivalently, to $t_{D}$) by $w_{D}$ (or just by
$w$), i.e., $ w_{D}:=\widetilde{v_{D}}=\widetilde{t_{D}}$.

Clearly
$w_{D}\leq t_{D}\leq v_{D}$. Moreover, from \cite[Gilmer (1972), Theorem
34.1(4)]{Gi}, we immediately deduce that $\ast \leq v_{D}$, and
thus $\widetilde{\ast }\leq w_{D}$ and $\ast_{_{\!f}}\leq
t_{D}$, for each  star operation $\ast$ on $D$.

Integral ideals that are  maximal with respect to being $\ast$--ideals, when $\ast =v$ or $t$ or $w$ are relevant in many situations. However, maximal $v$--ideals are not a common sight. There are integral
domains, such as a nondiscrete rank one valuation domain, that do not have
any maximal $v$--ideal  \cite[Gilmer (1972), Exercise 12, page 431]{Gi}.  Unlike maximal $v$--ideals, the maximal $t$--ideals are 
everywhere, in that every $t$--ideal is contained in at least one maximal 
$t$--ideal, which is always a prime ideal  \cite[Jaffard (1960), Corollaries 1 and 2, pages 30-31]{J:1960} (or, \cite[Malik (1979), Proposition 3.1.2]{Ma}, in the integral domains setting).  Note also that the set of maximal $t$--ideals coincides with  the set of maximal $w$--ideals   \cite[D.D. Anderson-Cook (2000), Theorem 2.16]{ACk}.

We will denote simply by $d_D$ (or just $d$) the \emph{identity star operation} on $D$ and clearly $d_D \leq \ast$, for each  star operation $\ast$ on $D$.   Another important star operation on an integrally closed domain  $D$ is the \emph{$b_D$--operation}  (or just \emph{$b$--operation}) defined as follows:
\begin{equation*}
A^{b_D}
:=\bigcap \left\{ AV\mid V \mbox{ is a valuation overring of } D \right\}\;\;\;\;  \mbox{\rm  for all $A \in
\boldsymbol{{F}}(D)$}.
\end{equation*}

Given a star operation on $D$, for $A\in \boldsymbol{{F}}(D)$, we say that $A$ is  \it
$\ast$--finite  \rm if there exists a $F\in \boldsymbol{f}(D)$
such that $F^{\ast}=A^{\ast}$. (Note that in the  above  
definition, we do not require that $F\subseteq A$.) It is
immediate to see that if $\ast _{1}\leq \ast _{2}$ are star
operations and $A$ is $\ast_{1}$--finite, then $A$ is $\ast_{2}$--finite. In particular, if $A$ is $\ast_{\!_{f}}$--finite, then it is $ \ast$--finite. The converse is
not true in general, and one can prove that $A$ is $\ast_{\!_{f}}$--finite if and only if there exists
 $F\in \boldsymbol{f}(D)$, $F\subseteq A$, such that $F^{\ast
}=A^{\ast}$  \cite[Zafrullah (1989), Theorem 1.1]{Z1989}.

Given a star operation on $D$, for $A\in \boldsymbol{{F}}(D)$, we say that $A$ is \it{$\ast
$--invertible} \rm if $(AA^{-1})^{\ast }=D$. From the
fact that the set of maximal $\widetilde{\ast}$--ideals,  $\text{Max}^{\widetilde{ \ast }}(D)$, coincides with the set of maximal ${\ast
_{\!_{f}}}$--ideals,  $\text{Max}^{\ast
_{\!_{f}}}(D)$, \cite[D.D. Anderson-Cook (2000), Theorem 2.16]{ACk}, it easily follows that a nonzero fractional ideal $A$ is
$\widetilde{\ast }$--invertible if and only if $A$ is $\ast
_{_{\!f}}$--invertible (note that if $ {\ast} $ is a star
operation of finite type, then $(AA^{-1})^{\ast} =D$ if
and only if $ AA^{-1}\not\subseteq Q$ for all $Q\in
\text{Max}^{{\ast} }(D)$). 

An   invertible ideal is a $\ast$--invertible $\ast $--ideal for any star operation $\ast $ and, in
fact, it is easy to establish that, if $\ast _{1}$ and $\ast _{2}$ are two
star operations on an integral domain $D$ with $\ast_1 \leq \ast_2$, then  any $\ast _{1}$--invertible ideal is also $\ast _{2}$--invertible.

A classical result due to Krull \cite[Jaffard (1960), Th\'eor\`eme 8, Ch. I, \S4]{J:1960} shows that for a
star operation $\ast$ of finite type,  $\ast$--invertibility implies $\ast$--finiteness. More precisely, for $A \in \F(D)$, we have  that
 $A$ is ${\ast
_{\!_{f}}}$--invertible if and only if $A$ and $A^{-1}$ are ${\ast
_{\!_{f}}}$--finite (hence, in
particular, $\ast$--finite) and $A$ is $\ast$--invertible   (see \cite[Fontana-Picozza (2005), Proposition 2.6]{FPi} for the semistar operation case).
\medskip

We recall now some notions and properties of monoid theory needed later. A nonempty set   with a binary associative and commutative law of composition ``\ $\cdot$\ ''  is  called a \emph{semigroup}. 
A \emph{monoid} $\KK$ is a semigroup that contains an identity element $\frak{1}$ (i.e., an element such that, 
 for all $x\in \KK$,  $\frak{1}\cdot x=x\cdot \frak{1}=x$). If there is an element $\frak{0}$ in 
$\KK$ such that, for all $x\in \KK$, $\frak{0}\cdot x=x\cdot \frak{0}=\frak{0}$, we say that $\KK$
has a \emph{zero element}. 
Finally if, for all $a,x,y$ in a monoid $\KK$ with $a\neq \frak{0},$ 
$
a\cdot x=a\cdot y$ implies that $x=y$ we say that $\KK$ is  a \emph{cancellative monoid}. In
what follows we shall be working with commutative and cancellative monoids
with or without zero.  Note that,
if $D$ is an integral domain then  $D$  can be considered as a monoid under
multiplication and, more precisely,  $D$ is a  cancellative monoid with zero element $0$.

Given a monoid $\KK$, we can consider the set of invertible elements in $\KK$, denoted by  $\U(\KK)$  (or, by $\KK^{\times}$) and the set $\KK^\bullet := \KK \setminus \{\frak{0}\}$.  Clearly, $\U(\KK)$ is a subgroup of (the monoid)  $\KK^\bullet$ and the monoid $\KK$ is called a \emph{groupoid} if $\U(\KK) = \KK^\bullet$.  A monoid with a unique invertible element is called \emph{reduced}. The monoid $\KK/\U(\KK)$ is reduced. A monoid  shall mean a reduced monoid unless specifically stated. 

Given a monoid $\KK$, we can easily develop a divisibility theory and we can introduce a GCD.    A \emph{GCD--monoid} is a monoid having a uniquely determined GCD for  each
finite  set of elements. In a monoid $\boldsymbol{\mathcal{H}}$ an element,     distinct from the unit element  $\mathfrak{1}$ and the zero element $\frak{0}$, is called \emph{irreducible} (or, \emph{atomic})  if it is divisible only by itself and $\mathfrak{1}$.   A monoid $\KK$ is called \emph{atomic} if every  nonzero noninvertible element  of $\KK$  is a product of finitely many atoms of $\KK$.  A nonzero noninvertible element  $p \in \KK$  with the property that $p\mid a\cdot b$, with $a, b \in \KK$ implies $p \mid a$ or $p \mid b$ is called a \emph{prime} element.  It is easy to see that in a GCD--monoid, irreducible and prime elements coincide.  

Given a monoid $\KK$, we can also
form the monoids of fractions of $\KK$ and, when $\KK$ is cancellative, the groupoid of fractions $\boldsymbol{q}(\KK)$ of $\KK$ in the
same manner, avoiding the zero element $\frak{0}$ in the denominator, as in the constructions of the rings of fractions and
the field of fractions of an integral domain $D$.

\bigskip
\centerline{$\diamond$\quad$\diamond$\quad$\diamond$}
\bigskip

This  survey paper   is the result of an effort to put together
information on the important class of integral
domains called $v$--domains, i.e., integral
 domains in which every finitely generated
nonzero  (fractional)   ideal is $v$--invertible. In the present work, we will use  a ring theoretic approach.
  However, because in multiplicative ideal theory we are mainly interested in the
multiplicative structure of the integral domains, the study of monoids came
into multiplicative ideal theory at an early stage. For instance, as we shall
indicate in the sequel, $v$--domains came out of a study of monoids. During
the second half of the 20th century, essentially due to the work of Griffin
\cite{Gr}, and due to Gilmer's books
\cite[Gilmer (1968)]{Gi} and 
 \cite[Gilmer (1984)]{Gi-84}, multiplicative ideal theory from a ring theoretic point of view
became a hot topic for the
  ring theorists. However, things appear to be
changing. Halter-Koch has put together in \cite[Halter-Koch (1998)]
{H-K}, in the
language of monoids, essentially all that was available  at that time and essentially all that could be translated
to the language of monoids.  On the 
other hand, more
recently,  Matsuda, under the influence
of \cite[Gilmer (1984)]{Gi-84},  is keen on converting 
into the language of
additive monoids and semistar operations all that is available and permits
conversion \cite[Matsuda (2002)]{Ma-02}.

Since translation of results often depends upon the interest, motivation and imagination of the
``translator'', it is a difficult task to indicate what (and in which way) can be translated into the language
of monoids, multiplicative or additive, or to the language of semistar
operations. But, one thing is certain, as we generalize, we gain a larger
playground but, at the same time, we lose the clarity and simplicity that we
had become so accustomed to.

With these remarks in mind, we indicate below some of
the results
that may or may not carry over to the monoid treatment, and we outline some general problems that can arise when looking for generalizations, without presuming to be exhaustive.
The first and foremost is any result to do with polynomial ring extensions
may not carry over to the language of monoids even though some of the
concepts translated to monoids do get used in the study of semigroup rings.
The other trouble-spot is the results on integral domains that use the identity ($d$--)operation. As soon as one considers the multiplicative monoid of an integral
domain, with or without zero, some things get lost. For instance,  the
multiplicative monoid $R\backslash \{0\}$ of a PID $R$, with more than one
maximal ideal, is no longer a principal ideal monoid, because a monoid has
only one maximal ideal, which in this case is not principal. All you can
recover is that $R\backslash \{0\}$ is a unique factorization monoid;
similarly, from a B\'ezout domain you can recover a GCD-monoid. Similar
comments can be made for Dedeking and Pr\"ufer domains. On the other hand, if
the $v$--operation is involved then nearly every result, other than the ones
involving polynomial ring extensions, can be translated to the language of
monoids. So, a majority of old ring theoretic results on $v$--domains and their
specializations can be found in \cite[Halter-Koch (1998)]{H-K} and some in \cite[Matsuda (2002)]{Ma-02}, in one form or another. We will mention or we will provide   precise references  only for
those results on monoids that caught our fancy for one reason or another, as indicated in the sequel.

The case of semistar operations and the possibility of generalizing results on $v$--domains, and their specializations,   in this setting is somewhat difficult in that the area of
research has only recently opened up \cite[Okabe-Matsuda (1994)]
 {OM2}.  Moreover, a number of
results involving semistar invertibility are now available, showing a more complex situation for the invertibility in the semistar operation setting see for instance \cite[Picozza (2005)]{Pi-05},  \cite[Fontana-Picozza (2005)]{FPi} and \cite[Picozza (2008)]
{Pi}.
 However, in studying semistar operations, in connection with $v$--domains,  
we often gain deeper insight, as recent work  indicates, see
\cite[D.F. Anderson-Fontana-Zafrullah (2008)]{AFZ}, 
\cite[Anderson-Anderson-Fontana-Zafrullah (2008)]
{AAFZ}.

\end{section} 

\medskip

\begin{section}{When and in what context did the $v$-domains \!show\! up?}

 
 \noindent \bf 2.a The genesis. \rm The \emph{$v$--domains} are precisely the integral domains $D$  for which the $v$--operation is an ``endlich arithmetisch brauchbar''  operation,  cf.  \cite[Gilmer (1968), page 391]{G68}.  
 Recall that a star operation $\ast$ on an integral domain $D$  is  \emph{endlich arithmetisch brauchbar}  (for short, \emph{e.a.b.}) (respectively,  \emph{arithmetisch brauchbar} (for short, \emph{a.b.})) if for all $F, G,H \in \boldsymbol{f}(D)$  (respectively, $F \in \boldsymbol{f}(D)$ and $  G,H \in \boldsymbol{F}(D)$)
$(FG)^{\ast}
             \subseteq (FH)^{\ast}$ implies that $G^{\ast}\subseteq H^{\ast} $.

   In \cite[Krull (1936)]{Krull:1936}, the author only considered the concept of 
 ``a.b. $\ast$--operation''  (more precisely,   Krull's original notation was  ``\ $^{\prime}$--Operation'', instead of  ``$\ast$--operation''). He did not  consider  the (weaker) concept of    ``e.a.b. $\ast$--operation''.  
 
The e.a.b. concept stems from the original version of
Gilmer's book \cite[Gilmer (1968)]{G68}. The results of Section 26 in \cite[Gilmer (1968)]{G68}  show that this 
(presumably) weaker
concept is all that one needs to develop a complete theory of Kronecker
function rings. 
       Robert Gilmer  explained to us   saying   that  \
{\small $\ll$}~I believe I was influenced to recognize this because
during the 1966 calendar year in our graduate algebra seminar (Bill
Heinzer, Jimmy Arnold, and Jim Brewer, among others, were in that
seminar) we had covered Bourbaki's Chapitres 5 and 7 of 
\it Alg\`ebre Commutative\rm , and the development in Chapter 7 on the $v$--operation indicated
that e.a.b. would be sufficient.~{\small{$\gg$}}

Apparently there are no examples in the literature of star operations which are e.a.b. but  not  a.b..  A forthcoming paper  \cite[Fontana-Loper-Matsuda (2010)]{FLoM} (see also \cite[Fontana-Loper (2009)]{FLo}) will contain  an explicit example to show that Krull's a.b. condition is really stronger than the Gilmer's e.a.b. condition.

\vskip 10pt

We asked Robert Gilmer and Joe Mott
about the origins of $v$--domains. They had the following to say: \
{\small $\ll$} We believe that Pr\"ufer's 
paper \cite[Pr\"ufer (1932)]{Prufer-32} is the first to discuss the concept in complete generality,  though we still do not know
who came up with the name of ``$v$--domain". {\small $\gg$}

However, the basic notion of $v$--ideal   appeared  around 1929. 
More precisely, the notion of quasi-equality of ideals (where, for $A, B \in \F(D)$, $A$ is \it quasi-equal \rm to $B$, 
if $A^{-1} = B^{-1}$), special cases
of $v$--ideals
and the observation that the classes of quasi-equal ideals of a   Noetherian   integrally closed domain form a group first appeared in \cite[van der Waerden (1929)]{vdW1} (cf. also \cite[Krull (1935), page 121]{Krull-35}), but this material was put into a more polished form by E. Artin and in this form was published for the first time  by Bartel Leendert van der Waerden in ``Modern Algebra''  \cite[van der Waerden (1931)]{vdW0}.   
 This book originated from notes taken by the author from E. Artin's lectures and it includes research of E. Noether and her students.  
   Note that the ``$v$'' of a $v$--ideal (or a $v$--operation) comes from the German ``Vielfachenideale'' or ``$V$--Ideale'' (``ideal of multiples"), terminology used in \cite[Pr\"ufer (1932), \S7]{Prufer-32}.  It is important to recall also the papers   \cite[Arnold (1929)]{Arnold-29} and  \cite[Lorenzen (1939)]{Lo-39} that introduce the study of $v$--ideals    and $t$--ideals   in semigroups.

 The paper \cite[Dieudonn\'{e} (1941)]{Di-41} provides a
clue to where $v$--domains came out as a separate class of rings, though they
were not called $v$--domains there. Note that \cite[Dieudonn\'{e} (1941)]{Di-41} has been
cited  in  \cite[Jaffard (1960), page 23]{J:1960} and, later,  in   \cite[Halter-Koch (1998), page 216]{H-K}, where it is mentioned
that J. Dieudonn\'{e} gives an example of a $v$--domain that is not a \emph{Pr\"ufer $v$--multiplication domain} (for short, \emph{P$v$MD},  i.e., an integral domain $D$ in which every $F\in \f(D)$ is $t$--invertible).

\bigskip



\noindent \bf 2.b Pr\"ufer domains and $v$--domains. \rm The $v$--domains generalize the \emph{Pr\"ufer domains} (i.e., the integral domains $D$ such that $D_M$ is a valuation domain for all $M \in \Max(D)$), since an integral domain $D$ is a
Pr\"ufer domain if and only if eve\-ry $F\in \f(D)$ is invertible \cite[Gilmer (1972), Theorem 22.1]{Gi}.  Clearly, an invertible ideal is $\ast$--invertible for all star operations $\ast$. In particular,  a Pr\"ufer domain is a \it Pr\"ufer $\ast$--multiplication domain \rm (for short, \emph{P$\ast$MD}, i.e., an integral domain $D$ such that,  for each $F \in \f(D)$,  $F$ is $\astf$--invertible \cite[Houston-Malik-Mott (1984), page 48]{HMM}). 
It is clear from the definitions that a P$\ast$MD is a P$v$MD  (since $\ast \leq v$ for all star operations $\ast$, cf. \cite[Gilmer (1972), Theorem 34.1]{Gi}) and a P$v$MD is a $v$--domain.

The picture can be refined.
M. Griffin,   a student of Ribenboim's,   showed that $D$ is a P$v$MD if and only if  $
D_{M}$ is a valuation domain for each maximal $t$--ideal $M$ of $D$ \cite[Griffin (1967), Theorem 5]{Gr}.   A generalization of this result is given in \cite[Houston-Malik-Mott (1984), Theorem 1.1]{HMM} by showing that $D$ is a P$\ast$MD if and only if  $
D_{Q}$ is a valuation domain for each maximal $\astf$--ideal $Q$ of $D$.   

Call 
a \emph{valuation overring $V$} of \emph{$D$ essential} if $V=D_{P}$ \ for some
prime ideal $P$ of $D$ (which is invariably the center of $V$ over $D$) and
call  $D$ an \textit{essential domain} if $D$ is expressible as an intersection of
its essential valuation overrings. Clearly,  a Pr\"ufer domain is essential
and so it  is a P$\ast$MD and, in particular, a P$v$MD (because, in the first case,  $D=\bigcap D_{Q}$ where $Q$ varies over maximal $\astf$--ideals of $D$ and $D_Q$ is a valuation domain;  in the second case, $D=\bigcap D_{M}$ where $M$ varies over maximal $t$--ideals of $D$ and $D_M$ is a valuation domain; see  \cite[Griffin (1967), Proposition 4]{Gr} and \cite[Kang (1989), Proposition 2.9]{Kang}).  

From a local point of view,   it is easy to see from the definitions that every integral
domain $D$ that is locally essential is essential.  The converse is not true and the first example of an essential domain having a prime ideal $P$ such that $D_P$ is not essential was given in  \cite[Heinzer (1981)]{He}.

Now add to this
information the following well known result \cite[Kang (1989), Lemma 3.1]{Kang-89} that shows that the essential domains sit in between   P$v$MD's   and $v$--domains.

\begin{proposition}
\label{ess-vdomain} An essential domain is a $v$-domain.
\end{proposition} 
\begin{proof}
Let $\Delta$ be a subset of $\Spec(D)$ such that $D$ = $\bigcap \{D_{P} \mid P \in \Delta \}$, where each $D_{P} $ is a valuation domain with center $P \in \Delta$,  let $F$ be a
nonzero finitely generated ideal of $D,$ and let $\ast_{\!\Delta} $ be the star
operation induced by the family of (flat) overrings  $\{D_{P} \mid P \in \Delta\}$ on $D$. Then 
$$
\begin{array}{rl}
(FF^{-1})^{\ast_{\!\Delta}}= &   \bigcap \{ (FF^{-1})D_{P } \mid P \in \Delta \}
=    \bigcap \{ FD_{P }F^{-1}D_{P } \mid P \in \Delta \} \\
= &    \bigcap \{ FD_{P }(F D_{P })^{-1} \mid P \in \Delta \} \quad \mbox{(because $F$ is f.g.)} \\   
 = &  \bigcap \{ D_{P }  \mid P \in \Delta \} \quad \mbox{(because $D_P$ is  a valuation domain).}
 \end{array}
 $$
 Therefore  $
(FF^{-1})^{\ast_{\!\Delta}}=D$  and so $(FF^{-1})^{v}=D$ (since $\ast_{\!\Delta} \leq v$ \cite[Gilmer (1972), Theorem 34.1]{Gi}).
\end{proof}

For an alternate   implicit proof of Proposition \ref{ess-vdomain}, and much more,
the reader may consult \cite[Zafrullah (1987), Theorem 3.1 and Corollary 3.2]{Z}. 

\begin{remark} \label{intersection-princ}
  (a)   Note   that Proposition \ref{ess-vdomain} follows also from a general result for essential monoids \cite[Halter-Koch (1998), Exercise 21.6 (i), page
244]{H-K},  but the result  as stated above (for essential
domains) was already known  for instance as an application of \cite[Zafrullah (1988), Lemma 4.5]{Z2}.   

If we closely look at  \cite[Halter-Koch (1998), Exercise 21.6, page
244]{H-K}, we note that part (ii)  was already known for the special case of integral domains (i.e., an essential domain is a P$v$MD if and only if the intersection of two principal ideals is a  $v$--finite $v$--ideal,  \cite[Zafrullah (1978), Lemma 8]{Zi}) 
and   part (iii) is related to the following fact concerning integral domains:\  for $F \in \f(D)$, $F$ is $t$--invertible if and only if  $(F^{-1} :F^{-1}) =D$ and $F^{-1}$ is $v$--finite.  The previous property follows immediately from the following statements: 
\begin{itemize} \it
\item[\rm (a.1)] \hskip 20pt let $F\in \f(D)$, then $F$ is $t$--invertible if and only if $F$ is $v$--invertible and $F^{-1}$ is $v$--finite;  
\item[\rm (a.2)] \hskip 20pt  let $A \in \F(D)$, then $A$ is $v$--invertible if and only if $(A^{-1} : A^{-1}) =D$.
\end{itemize}

 The statement  (a.1)  can be found in \cite[Zafrullah (2000)]{Zaf}  and   (a.2)  is posted in \cite[Zafrullah (2008)]{helpdesk}. For reader's convenience, we  next  give their proofs.   

For the  ``only if part''  of (a.1),   if $F\in \f(D)$ is  $t$--invertible, then $F$ is clearly $v$--invertible and $F^{-1}$ is also $t$--invertible. Hence $F^{-1}$ is $t$--finite and thus $v$--finite.

For a ``semistar version'' of  (a.1),  see for instance \cite[Fontana-Picozza (2005), Lemma 2.5]{FPi}.

For the   ``if part''   of  (a.2), note that $AA^{-1} \subseteq D$ and so   $(AA^{-1})^{-1} \supseteq D$. 
Let $x \in (AA^{-1})^{-1} $, hence $xAA^{-1} \subseteq D$ and so $xA^{-1} \subseteq A^{-1}$, i.e., $x \in (A^{-1} : A^{-1}) =D$.
For the   ``only if part'',   note that in general $D \subseteq (A^{-1} : A^{-1})$. For the reverse inclusion, let $x \in (A^{-1} : A^{-1})$, hence $xA^{-1} \subseteq A^{-1}$. Multiplying both sides by $A$ and applying the $v$--operation, we have $x D = x(AA^{-1})^{v} \subseteq (AA^{-1})^{v} =D$, i.e., $x \in D$ and so  $D \subseteq (A^{-1} : A^{-1})$.   A simple proof of (a.2) can also be deduced from
\cite[Halter-Koch (1998), Theorem 13.4]{H-K}.

 It is indeed
remarkable that all those results known for integral domains can be
interpreted and extended to monoids.

 (b)   We have observed in  (a)   that a P$v$MD is an essential domain such that the intersection of two principal ideals is a  $v$--finite   $v$--ideal.  It can be also shown that  $D$ is a P$v$MD 
if and only if $(a) \cap (b)$ is $t$--invertible in $D$,  for all nonzero $a, b \in D$  \cite[Malik-Mott-Zafrullah (1988), Corollary 1.8]{MMZ}.

For $v$--domains we have the following ``$v$--version'' of the previous characterization for P$v$MD's: 
$$
D \mbox{ \it is a $v$--domain} \, \Leftrightarrow \, (a) \cap (b) \mbox{ \it is $v$--invertible in $D$,  for all nonzero $a, b \in D$.}$$

 The idea of proof is simple and goes along the same lines as those of P$v$MD's.  Recall that  every $F \in \f(D)$  is invertible (respectively, $v$--invertible; $t$--invertible) if and only if every nonzero two generated ideal of $D$ is invertible (respectively, $v$--invertible; $t$--invertible) \cite[Pr\"ufer (1932), page 7]{Prufer-32} or  \cite[Gilmer (1972), Theorem 22.1]{Gi}  (respectively,  for the ``$v$--invertibility case'',   \cite[Mott-Nashier-Zafrullah (1990), Lemma 2.6]{MNZ};    for the ``$t$--invertibility case'',   \cite[Malik-Mott-Zafrullah (1988), Lemma 1.7]{MMZ});   for the general case of star operations, see the following  Remark \ref{rk:5} (c).  Moreover, for all nonzero $a, b \in D$, we have:
$$
\begin{array}{rl}
(a, b)^{-1} =& \frac{1}{a}D \cap \frac{1}{b}D  = \frac{1}{ab}(aD \cap bD)\,, \\
(a,b)(a, b)^{-1} =&  \frac{\, 1\, }{\, ab\, }(a, b)(aD \cap bD)\,.
\end{array}
$$
Therefore, in particular, the fractional ideal $(a,b)^{-1}$ (or, equivalently, $(a,b)$)  is   $v$--invertible  if and only if the ideal $aD \cap bD$ is  $v$--invertible.

  (c)    Note that, by the observations contained in the previous point  (b), if $D$ is a Pr\"ufer domain 
then  $(a) \cap (b)$ is  invertible in $D$,  for all nonzero $a, b \in D$. However, the converse is not true,  as we will see in Sections 2.c
 and 2.e 
  (Irreversibility of $\Rightarrow _{7}$). The reason for this is that   $aD \cap bD$ invertible allows only that the ideal
   $\frac{\ (a, b)^v}{ab}$ (or, equivalently, ${(a, b)}^v$) is invertible and not necessarily the ideal $(a,b)$. \end{remark} 

 Call a \it P-domain \rm an integral  domain such that every ring of fractions  is essential (or, equivalently, \emph{a locally essential domain}, i.e., an integral domain $D$ such that $D_P$ is essential, for each prime ideal $P$ of $D$) \cite[Mott-Zafrullah (1981), Proposition 1.1]{MZ}. Note that every ring of fractions of a P$v$MD is still a P$v$MD (see Section \ref{fractions} for more details), in particular, since a P$v$MD is essential, a locally P$v$MD is a P-domain.
Examples of P-domains include  Krull domains. As a matter of fact, by using Griffin's characterization of P$v$MD's \cite[Griffin (1967), Theorem 5]{Gr}, a Krull domain is a P$v$MD, since in a Krull domain $D$ the maximal $t$--ideals (= maximal $v$--ideals)  coincide with the height 1 prime ideals \cite[Gilmer (1972), Corollary 44.3 and 44.8]{Gi} and  $D = \bigcap \{D_P \mid P \mbox{ is an height 1 prime ideal of $D$} \}$,  where $D_P$ is a discrete valuation domain for all height 1 prime ideals $P$ of $D$ \cite[Gilmer (1972), (43.1)]{Gi}.   Furthermore, it is well known that every ring of fractions of a Krull domain is still a Krull domain \cite[BAC, Ch. 7, \S 1, N. 4, Proposition 6]{Bo}.
  
  With these observations at hand,  we have the following picture:

$$
\begin{array}{rl} 
\mbox{Krull domain} \; \Rightarrow _{0}  &  \mbox{P$v$MD}\,; \\
\mbox{Pr\"ufer domain} \; \Rightarrow _{1}  &  \mbox{P$v$MD} \; \Rightarrow _{2}  \; \mbox{locally P$v$MD}  \\
 \Rightarrow
_{3} &  \mbox{P-domain} \; \Rightarrow _{4} \; \mbox{essential domain} \; \\ \Rightarrow
_{5}&  \mbox{$v$--domain}\, .
\end{array}
$$

 \begin{remark} \label{rk:3} \rm Note that  P-domains were  originally  defined as the integral domains $D$ such that $D_Q$ is a valuation domain for every \emph{associated prime ideal $Q$ of a principal ideal of $D$} (i.e., for every prime ideal which is minimal over an ideal of the type $(aD : bD)$ for some $a \in D$ and $b \in D \setminus aD$) \cite[Mott-Zafrullah (1981), page 2]{MZ}.  The P-domains were characterized in a somewhat special way
 in \cite[Papick (1983), Corollary 2.3]{Pa}:  $D$ is a P-domain if and only if $D$ is integrally closed and, for each $u \in K$, $D \subseteq D[u]$ satisfies INC  at every associated prime ideal $Q$ of a principal ideal of $D$.
  \end{remark}


\noindent \bf 2.c B\'ezout-type domains and $v$--domains. \rm Recall that an integral domain $D$ is a \it B\'ezout domain \rm if every
finitely generated ideal of $D$ is principal and $D$ is a \emph{GCD domain} if, for all nonzero $a, b \in D$,  a greatest common divisor of $a$ and $b$,  GCD$(a,b)$, exists   and is in $D$.  Among the characterizations of the GCD domains we have that $D$ is a GCD domain if and
only if, for every $F\in \f(D)$,  $F^{v}$ is principal or, equivalently, if and only if the intersection of two (integral) principal ideals of $D$ is still principal (see, for instance, \cite[D.D. Anderson (2000), Theorem 4.1]{Anderson-00}   and also Remark \ref{intersection-princ} (b)).   From Remark \ref{intersection-princ} (b), we deduce immediately that a GCD domain is a $v$--domain. 

 However, in between GCD domains and $v$--domains lie several other distinguished classes of integral domains.    An important generalization of the notion of GCD domain was introduced in \cite[Anderson-Anderson (1979)]{AA} where an integral domain $D$ is called a  \it Generalized GCD \rm (for short, \emph{GGCD})
\emph{domain}  if the intersection of two (integral) invertible ideals of $D$ is invertible $D$. 
 It is well known that $D$ is a  GGCD   domain  if and only if, for each $F\in \f(D)$,  $F^{v}$ is invertible \cite[Anderson-Anderson (1979), Theorem 1]{AA}.   In particular, a Pr\"ufer domain is a   GGCD domain. 
 From the fact that  an invertible ideal in a local domain is principal \cite[Kaplansky (1970), Theorem 59]{Kap}, we easily deduce that a GGCD domain is locally a GCD domain.
  On the other hand, from the definition of P$v$MD, we easily deduce that a GCD domain is a P$v$MD
 (see also \cite[D.D. Anderson (2000), Section 3]{Anderson-00}).  Therefore,  we have the following addition to the existing picture: 
$$
\begin{array}{rl}
\mbox{B\'ezout domain} \; \Rightarrow _{6}  &   \mbox{GCD domain} \; \Rightarrow _{7}  \; \mbox{GGCD domain}  \\
 \Rightarrow
_{8} & \mbox{locally GCD domain} \; \Rightarrow _{9} \; \mbox{locally P$v$MD} \; \\ \Rightarrow
_{3}&  \mbox{.....} \; \Rightarrow _{4}  \mbox{.....} \;  \Rightarrow _{5}  \mbox{$v$--domain}\, .
\end{array}
$$



\noindent \bf 2.d Integral closures and $v$--domains. \rm Recall  that  an integral domain $D$ with quotient field $K$  is called a \emph{completely integrally closed} (for short, \emph{CIC}) \emph{domain} if $D =    \{z \in K \mid \mbox{ for all $n \geq 0$, } az^n \in D \mbox { for some nonzero $a \in D$}\}$.    It is well known that \it the following statements are equivalent.
\begin{itemize}
\item[\rm (i)]  \hskip 10pt $D$ is  CIC;

\item[\rm (ii)] \hskip 10pt  for all $A\in \F(D)$,  $(A^v: A^v)=D$; 

\item[\rm (ii$^\prime$)] \hskip 10pt  for all $A\in \F(D)$,  $(A:A)=D$;

 \item[\rm (ii$^{\prime\prime}$)] \hskip 10pt   for all $A\in \F(D)$,  $(A^{-1}:A^{-1})=D$; 

\item[\rm (iii)] \hskip 10pt   for all $A\in \F(D)$, $(AA^{-1})^{v}=D$;

\end{itemize}\rm
(see \cite[Gilmer (1972), Theorem 34.3]{Gi} and Remark \ref{intersection-princ} (a.2); for a general monoid version of this characterization, see \cite[Halter-Koch (1998), page 156]{H-K}). 

In Bourbaki \cite[BAC, Ch. 7, \S 1,   Exercice 30]{Bo} an
integral domain $D$ is called \emph{regu\-larly integrally closed} if,  for all $F \in
\f(D)$, $F^v$ is regular with respect to the $v$--multiplication (i.e., if $(FG)^v = (FH)^v$ for  $G, H \in \f(D)$ then $G^v = H^v$). 

\begin{theorem} \label{lorenzen}\rm  (\cite[Gilmer (1972), Theorem 34.6]{Gi} and \cite[BAC, Ch. 7, \S 1,  Exercice  30 (b)]{Bo}) \it Let $D$ be an integral domain, then the following are equivalent.
\begin{itemize}
\item[\rm (i)] \hskip 10pt  $D$ is a regularly integrally closed domain.
\item[\rm (ii$_{\f}$)]  \hskip 10pt  For all $F \in
\f(D)$, $(F^{v}: F^{v})=D$. 

\item[\rm (iii$_{\f}$)] \hskip 10pt   For all $F\in \f(D)$ $
(FF^{-1})^{-1}=D$  (or, equivalently, $
(FF^{-1})^v=D$).

\item[\rm (iv)]  \hskip 10pt  $D$ is a $v$--domain.
\end{itemize}
\end{theorem}
  The original version of Theorem \ref{lorenzen} appeared in \cite[Lorenzen (1939), page 538]{Lo-39} (see also \cite[Dieudonn\'{e} (1941), page 139]{Di-41}  and \cite[Jaffard (1951), Th\'eor\`eme 13]{J:1951}).     A general monoid version of the previous cha\-racterization is given in  \cite[Halter-Koch (1998), Theorem 19.2]{H-K}.

  \begin{remark}\label{rk:5} (a) Note that the condition 
 
 \noindent (ii$'_{\f}$)  \emph{for all} $F \in
\f(D)$, $(F: F)=D$ \\
is equivalent to say that $D$ is integrally closed \cite[Gilmer (1972), Proposition 34.7]{Gi} and so it is weaker than condition (ii$_{\f}$) of the previous Theorem \ref{lorenzen}, since $(F^{v}: F^{v}) = (F^{v}:  F) \supseteq (F:F)$.

 On the other hand, by Remark \ref{intersection-princ} (a.2),
  the condition
 
 \noindent (ii$^{\prime\prime}_{\f}$)   \emph{for all } $F \in
\f(D)$, $(F^{-1}: F^{-1})=D$ \\
is equivalent to the other statements  of  Theorem \ref{lorenzen}.  

 (b)    By \cite[Mott-Nashier-Zafrullah (1990), Lemma 2.6]{MNZ}, condition (iii$_{\f}$) of the previous theorem   is equivalent to

\noindent (iii$_{\boldsymbol{2}}$) \emph{Every nonzero fractional ideal with two generators is 
$v$--invertible}.

  This characterization is a variation of  Pr\"ufer's classical result that an integral domain is Pr\"ufer if and only if each nonzero ideal with two generators is invertible  (Remark \ref{intersection-princ} (b)) and of the characterization of P$v$MD's also  recalled in that remark.

 (c)   Note that several classes of Pr\"ufer-like domains can be studied in a unified frame by using star and semistar operations. For instance Pr\"ufer star-multiplication domains were introduced  in \cite[Houston-Malik-Mott (1984)]{HMM}. Later, in \cite[Fontana-Jara-Santos (2003)]{FJS},
 the authors studied  Pr\"ufer semistar-multiplication domains and gave several   characterizations of these domains, that are new also for the classical case of P$v$MD's. Other important contributions, in general settings, were given recently in   \cite[Picozza (2008)]{Pi} and \cite[Halter-Koch (2008)]{H-K3}.    \\
 In \cite[Anderson-Anderson-Fontana-Zafrullah (2008), Section 2]{AAFZ},  given a star operation $\ast$ on an integral domain $D$, the authors call   $D$ a  \it   $\ast$--Pr\"ufer domain \rm       if every nonzero finitely generated  ideal of $D$ is 
$\ast $--invertible  (i.e.,  $(FF^{-1})^{\ast}=D$ for  all $F\in \f(D)$).      (Note that    $\ast$--Pr\"ufer  domains  were previously introduced  in  the  case of semistar operations $\star$ under the name of $\star$--domains \cite[Fontana-Picozza (2006), Section 2]{FPi2}.)
 Since  a $\ast $--invertible ideal  is always $v$--invertible, a   $\ast$--Pr\"ufer domain   is always a $v$--domain. More precisely, $d$--Pr\"ufer (respectively,  $t$--Pr\"ufer; $v$--Pr\"ufer) domains coincide with  Pr\"ufer (respectively, Pr\"ufer $v$--multiplication; {$v$--~\!)} domains.
    
    Note that, in  \cite[Anderson-Anderson-Fontana-Zafrullah (2008), Theorem 2.2]{AAFZ}, the authors 
 show that a star operation version of (iii$_{\boldsymbol{2}}$) considered in point (b) characterizes $\ast$--Pr\"ufer domains, i.e., $D$ is a $\ast$--Pr\"ufer domain if and only if  every nonzero two generated  ideal of $D$ is 
$\ast $-invertible.     An analogous result, in the general setting of monoids,  
can be found in  \cite[Halter-Koch (1998), Lemma 17.2]{H-K}. 

 (d)    Let $\f^v(D) := \{ F^v\mid F \in \f(D)\}$  be the set of all divisorial ideals of finite type of an integral domain $D$ (in \cite[Dieudonn\'{e} (1941)]{Di-41}, this set is denoted by $\mathfrak{M}_{f}$).  By Theorem \ref{lorenzen}, we have that a $v$--domain is an integral domain  $D$ such that each element of $F^v \in  \f^v(D)$ is $v$--invertible,   but  $F^{-1} \ (= (F^v)^{-1})$ does not necessarily belong to $\f^v(D)$.   When (and only when), in a $v$--domain $D$,  $F^{-1} \in \f^v(D)$,   $D$ is a P$v$MD (Remark \ref{intersection-princ} (a.1)).   

The ``regular'' teminology for the elements of $\f^v(D)$  used by \cite[Dieudonn\'{e} (1941), page 139]{Di-41} (see the above definition of  $F^v$ regular with respect to the $v$--multiplication)   is totally
different from the  notion of  ``von Neumann regular'', usually considered for elements of a ring or of a semigroup.  However, it may be instructive to record some observations showing that, in the present situation, the two notions are somehow related.

 Recall that, by a \emph{Clifford semigroup}, we mean a multiplicative commutative semigroup $\KK$, containing a unit element, such that each
element $a$ of $\KK$ is \emph{von Neumann regular} (this means that  there is $b\in \KK$ such
that $a^{2}b=a$).

\begin{itemize}
\item[$(\alpha)$] \hskip 10pt
 Let $\KK$ be a commutative and cancellative monoid.   If   $\KK$ is a  Clifford semigroup, 
 then $a$ is invertible in $\KK$ (and conversely); in other words, $\KK$ is a group.  

\item[$(\beta)$] \hskip 10pt Let $D$ be a $v$--domain.
If  $A\in \f^v(D)$ is von Neumann
regular in the monoid  $\f^v(D)$ under $v$--multiplication, then $A$ is $t$--invertible (or, equiva\-lently, $A^{-1} \in \f^v(D)$). Consequently,  an integral domain $D$  is a P$v$MD if and only if $D$ is a $v$--domain and the monoid  $
\f^v(D)$ (under $v$--multiplication) is
Clifford regular.  

\end{itemize}

 The proofs of $(\alpha)$ and $(\beta)$ are straightforward, after recalling that $\f^v(D)$ under $v$--multiplication is a commutative monoid and, by definition, it is cancellative if  $D$ is a $v$--domain.

\smallskip

Note that, in   the ``if part'' of   $(\beta)$, the assumption that $D$ is   a $v$--domain is essential.
As a matter of fact, it is not true that  an integral domain $D$, such that every member of the monoid  $\f^v(D)$ under the $v$--operation  is von
Neumann regular, is  a $v$--domain. 
For instance,  in \cite[Zanardo-Zannier (1994), Theorem 11]
{ZZ}  (see also \cite[Dade-Taussky-Zassenhaus (1962)]{DTZ}),  the authors show that for every quadratic order $D$,  each nonzero ideal $I$ of $D$ satisfies  $I^{2}J=cI
$, i.e., $I^{2}J(1/c)=I$,  
 for some (nonzero) ideal $J$ of $D$ and some nonzero $c\in D$. So, in particular,  in this situation $\f(D) = \F(D)$ and  every element of the monoid  $\f^v(D)$ is von Neumann regular (we do not even need to apply the $v$--operation
in this case), however not  all quadratic orders are integrally closed (e.g., $D :=\mathbb Z[\sqrt 5]$) and so, in general,  not all elements of  $\f^v(D)$ are regular with respect to the $v$--operation (i.e., $D$ is not a $v$--domain).

Clifford regularity  for  class and $t$--class semigroups of ideals in  various types of integral domains   was investigated, for instance, in [20 and 21, Bazzoni (1996), (2001)]
\cite[Fossum (1973)]{Fuchs}, [71,72 and 73, Kabbaj-Mimouni, (2003), (2007), (2008)],
 \cite[Sega (2007)]{Se}, and [54 and 55, Halter-Koch (2007), (2008)].
 In particular, in the last paper, Halter-Koch proves a stronger and   much  deeper version of $(\beta)$, that is, a $v$--domain  having its $t$--class semigroups of ideals Clifford regular is a domain of Krull-type (i.e., a P$v$MD  with finite $t$--character). This result generalizes \cite[Kabbaj-Mimouni (2007), Theorem 3.2]{KM2} on Pr\"ufer $v$--\-multi\-pli\-cation domains.  

 (e)     In the situation of point (d, $\beta$), 
the condition that every   $v$--finite $
v$--ideal     is regular, in the sense of von Neumann,   in the larger monoid $
\F^v(D) := \{A^v \mid A \in \F(D)\}$ of all $v$--ideals of    $D$ (under $v$--multiplication) is too weak  to imply that $D$ is a P$v$MD.  

As a matter of fact, if we assume that $D$ is a $v$--domain, then  every $A \in \f^v(D)$ is $v$--invertible  in the (larger) monoid $\F^v(D)$. Therefore, $A$ is von Neumann regular in $\F^v(D)$, since $(AB)^v = D$ for some $B \in \F^v(D)$ and thus, multiplying both sides by $A$ and applying the $v$--operation,  we get   $(A^2B)^v = A $.

\end{remark}

\begin{remark} \label{pseudo} \rm  Regularly
integrally closed integral domains make their appearance with a different terminology  in the study of a weaker form of
 integrality,   introduced in the paper  \cite[D.F. Anderson-Houston-Zafrullah (1991)]{AHZ}. Recall that, given an integral domain $D$ with quotient field $K$,  an
element $z \in K$ is called \emph{ pseudo-integral over}  $D$ if $z\in (
F^{v}: F^{v})$ for some $F\in \f(D)$.  The terms  \emph{pseudo-integral closure} (i.e., $
\widetilde{D}:=\bigcup \{(F^{v}: F^{v})) \mid F \in \f(D)\}$ and \emph{pseudo-integrally closed domain} (i.e., $D=\widetilde{D}$) are coined in the obvious fashion and
it is clear from the definition that  pseudo-integrally closed coincides with regularly integrally closed.
\end{remark}

 From the previous
observations,  we have the following addition to the existing picture:  
$$
 \mbox{CIC domain} \; \Rightarrow _{10}  \; \mbox{ $v$--domain} \; \Rightarrow _{11}  \; \mbox{integrally closed domain}.
$$

  Note that in the Noetherian case, the previous  three classes of domains coincide (see the following Proposition \ref{pr:8} (2) or \cite[Gilmer (1972), Theorem 34.3 and Proposition 34.7]{Gi}). Recall also  that Krull domains can be characterized by the property that, for all $A\in \F(D)$,
$A$ is $t$--invertible   \cite[Kang (1989), Theorem 3.6]{Kang-89}.  This property is clearly  stronger than the  condition (iii$_{\f}$) of previous Theorem \ref{lorenzen}  and, more precisely, it is strictly stronger than (iii$_{\f}$),  since a Krull domain is CIC  (by condition (iii) of the above characterizations of CIC domains, see also \cite[BAC, Ch.7, \S 1, N. 3, Th\'eor\`eme 2]{Bo}) and a CIC domain is a $v$--domain, but the converse does not hold, as we will see in the following Section 2.e.

\begin{remark} \rm
Note that Okabe and
Matsuda \cite[Okabe-Matsuda (1992)]{OM} generalized pseudo-integral closure to the star operation setting. Given a star operation $\ast$ on an integral domain $D$, they  call the \emph{$\ast $--integral
closure of} $D$ its overring   $\bigcup \{(F^{\ast }: F^{\ast }) \mid F\in \f(D)\}$ denoted
by $\cl^{\ast }(D)$ in \cite[Halter-Koch (1997)]{H-K1}.   
Note   that, in view of this notation, $\widetilde{D}
=\cl^{v}(D)$ (Remark \ref{pseudo}) and the integral closure $\overline{D}$ of $D$ coincides with $\cl^{d}(D)$ \cite[Gilmer (1972), Proposition 34.7]{Gi}.  Clearly, if $\ast_1$ and $ \ast_2$  are two star operations on $D$ and $\ast_1 \leq \ast_2$, then 
$\cl^{\ast_1 }(D) \subseteq \cl^{\ast_2}(D)$. In particular, for each star operation $\ast$ on $D$, we have
$\overline{D} \subseteq \cl^{\ast }(D) \subseteq \widetilde{D}$.

It is not hard to see that $\cl^{\ast }(D)$ is integrally closed  \cite[Okabe-Matsuda (1992), Theorem 2.8]{OM} and is contained in the complete integral closure of $D$, which coincides with $\bigcup \{(A: A) \mid A\in \F(D)\}$ \cite[Gilmer (1972), Theorem 34.3]{Gi}.

Recall also that, in \cite[Halter-Koch (1998), Section 3]{H-K}, the author introduces a star operation of finite type on the integral domain $\cl^{\ast }(D)$, that we denote here by $\cl(\ast)$, ~defined as follows, for all $G \in \f(\cl^{\ast}(D))$:
$$ 
G^{
{\mbox{\footnotesize{\cl}}}
(\ast)} : = \bigcup \{((F^\ast: F^\ast)G)^\ast \mid F \in \f(\cl^{\ast}(D)) \}\,.
$$
Clearly, $D^{
{\mbox{\footnotesize{\cl}}}(\ast)} = \cl^{\ast }(D)$.  Call an integral domain  $D$ \emph{$\ast$--integrally closed}  when $D = \cl^{\ast }(D)$.  Then, from the fact that $\cl(\ast)$ is a star operation on $\cl^{\ast }(D)$,  it follows that $\cl^{\ast }(D)$ is
$\cl(\ast)$--integrally closed.   In general, if $D$ is not necessarily $\ast$--integrally closed, then $\cl(\ast)$, defined on $\f(D)$,   gives rise naturally  to a  semistar operation (of finite type) on $D$ \cite[Fontana-Loper (2001), Definition 4.2]{FLo-01}. 

  Note that   the domain $\widetilde{D} \ (= \cl^v(D))$, even if   it is $\cl(v)$--integrally closed,  in general  is not $v_{\widetilde{D}}$--integrally closed; a counterexample is given in \cite[D.F. Anderson-Houston-Zafrullah (1991), Example 2.1]{AHZ} by using a construction due to \cite[Gilmer-Heinzer (1966)]{Gi-He}. On the other hand, since an integral domain $D$ is a $v$--domain if and only if $D =\cl^v(D)$ (Theorem \ref{lorenzen}), from the previous observation we deduce that, in general, $\widetilde{D} $ is not a $v$--domain.   On the other hand,    using a particular ``$D+M$ construction'', in \cite[Okabe-Matsuda (1992), Example 3.4]{OM}, the authors construct an example of a non--$v$--domain $D$ such that $\widetilde{D}$ is a $v$--domain, i.e., $D \subsetneq \widetilde{D}=\cl^{v_{\widetilde{D}}}(\widetilde{D})$.
\end{remark}


 \noindent \bf 2.e Irreversibility of the implications ``${\boldsymbol{\Rightarrow _{n}}}$''. \rm  We start by observing that, under standard finiteness assumptions, several classes of domains considered above coincide.   Recall that an integral domain $D$ is called \emph{$v$--coherent} if   a finite  intersection of $v$--finite $v$--ideals is a $v$--finite $v$--ideal  or, equivalently, if $F^{-1}$ is $v$--finite for all $F \in \f(D)$ \cite[Fontana-Gabelli (1996), Proposition 3.6]{FoGa}, and it is called a \emph{$v$--finite conductor domain} if the intersection of two principal ideals is $v$--finite   \cite[Dumitrescu-Zafrullah (2008)]{DZ}.  From the definitions, it follows that  a $v$--coherent domain is a $v$--finite conductor domain.
From Remark \ref{intersection-princ} (a.1),  we deduce immediately that
$$
D \mbox{ \it is a P$v$MD} \;  \Leftrightarrow \; D \mbox{ \it is a $v$--coherent $v$--domain.}
$$
 In case of a $v$--domain, the notions of $v$--finite conductor domain and $v$--coherent domain coincide.  As a matter of fact,   as we have observed in Remark \ref{rk:5} (c), a  P$v$MD is exactly a $t$--Pr\"ufer domain and an integral domain $D$ is $t$--Pr\"ufer if and only if every nonzero two generated ideal is $t$--invertible. This translates to $D$ is a P$v$MD if and only if $(a,b)$ is $v$--invertible and $(a) \cap (b)$ is $v$--finite, for all $a, b \in D$ (see also Remark \ref{rk:5} (b)). In other words, 
 $$
D \mbox{ \it is a P$v$MD} \;  \Leftrightarrow \; D \mbox{ \it is a $v$--finite conductor $v$--domain.}
$$

Recall that an integral domain  
 $D$ is a GGCD domain if and only if  $D$ is a P$v$MD that is a
locally GCD domain  \cite[Anderson-Anderson (1979), Corollary 1 and page 218]{AA} or \cite[Zafrullah (1987), Corollary 3.4]{Z}. On the other hand, we have already observed that a locally GCD domain  is essential and it is known that an essential $v$--finite conductor domain is a P$v$MD \cite[Zafrullah (1978), Lemma 8]{Zi}. 
The situation is summarized in the following:

\begin{proposition} \label{pr:8} Let $D$ be an integral domain.
\begin{itemize}
\item[\rm (1)] \hskip 6pt Assume that $D$ is a $v$--finite conductor (e.g., Noetherian) domain. Then, the following classes of domains coincide:
{
\begin{itemize}
\item[\rm (a)] \hskip 6pt P$v$MD's;
\item[\rm (b)] \hskip 6pt  locally  P$v$MD's;
\item[\rm (c)]\hskip 6pt  P--domains;
\item[\rm (d)]\hskip 6pt  essential domains.
\item[\rm (e)] \hskip 6pt locally $v$--domains;
\item[\rm (f)] \hskip 6pt $v$--domains.

\end{itemize}
}
\item[\rm (2)] \hskip 6pt  Assume that $D$ is a Noetherian domain.  Then,  the previous classes of domains \rm (a)--(f) \it coincide also with the following:
{
\begin{itemize}
\item[\rm (g)] \hskip 6pt Krull domains;
\item[\rm (h)] \hskip 6pt CIC domains;
\item[\rm (i)] \hskip 6pt integrally closed domains.
\end{itemize}
}

\item[\rm (3)] \hskip 6pt Assume that $D$ is a $v$--finite conductor (e.g., Noetherian) domain. Then, the following classes of domains coincide:
{
\begin{itemize}

\item[\rm (j)] \hskip 6pt GGCD domains;
\item[\rm (k)] \hskip 6pt locally GCD domains.
\end{itemize} 
}
\end{itemize}
\end{proposition} 

 Since the notion of Noetherian B\'ezout (respectively, Noetherian GCD) domain coincides with the notion of PID or principal ideal domain (respectively, of Noethe\-rian UFD (= unique factorization domain)  \cite[Gilmer (1972), Proposition 16.4]{Gi}), in the Noetherian case the picture of all classes considered above reduces to the following:
$$
\begin{array}{rl}
\mbox{Dedekind domain} &  \Rightarrow_{1,2,3,4,5}     \mbox{ \ $v$--domain} \\ 
\mbox{PID} & \Rightarrow_{6}  \mbox{ \ UFD} \;\Rightarrow_{7,8} \;  \mbox{locally UFD} \;\Rightarrow_{9,3,4,5}  \mbox{$v$--domain.}
\end{array}$$

In general, of   the implications \ $\Rightarrow _{n}$ \  (with  $0 \leq n\leq 11$) discussed above all, except \ $\Rightarrow _{3}$\ ,  are known to be irreversible. We leave
the case of irreversibility of  \ $\Rightarrow _{3}$ \ as an open question and
proceed to give examples to show that all the other implications are
irreversible.

$\bullet$ Irreversibility of $\Rightarrow _{0}$. Take any nondiscrete valuation domain or, more generally, a Pr\"ufer non-Dedekind domain.

$\bullet$ Irreversibility of $\Rightarrow _{1}$  (even in the Noetherian case).    Let $D$ be a
Pr\"ufer domain that is not a field and let $X$ be an indeterminate over $D$. Then, as $D[X]$ is a P$v$MD if and only if  $D$ is \cite[S. Malik (1979), Theorem 4.1.6]{Ma} (see also \cite[D.D. Anderson-D.F. Anderson (1981), Proposition 6.5]{AA:81}, \cite[B.G. Kang (1989), Theorem 3.7]{Kang},
\cite[D.D. Anderson-Kwak-Zafrullah (1995), Corollary 3.3]{AKZ}, and  the following Section \ref{poly}), we conclude that $D[X]$ is a P$v$MD that is not Pr\"ufer. An explicit example is $\mathbb Z[X]$, where $\mathbb Z$ is the ring of integers.

$\bullet$ Irreversibility of $\Rightarrow _{2}$.   It is well known   that
every ring of fractions of a P$v$MD is again a P$v$MD \cite[Heinzer-Ohm (1973), Proposition 1.8]{HO} (see also the following Section \ref{fractions}). The fact that $\Rightarrow _{2}$ is
not reversible has been shown by producing examples of locally P$v$MD's that
are not P$v$MD's. In \cite[Mott-Zafrullah (1981), Example 2.1]{MZ} an example of a non P$v$MD
essential domain due to Heinzer and Ohm \cite[Heinzer-Ohm (1973)]{HO} was shown to have the
property that it was locally P$v$MD and hence a P-domain.

$\bullet$ Irreversibility of $\Rightarrow _{3}$:  Open. However, as mentioned above, \cite[Mott-Zafrullah (1981), Example 2.1]{MZ} shows the existence of a P-domain which is not a P$v$MD. Note that   \cite[Zafrullah (1988), Section 2]{Z2}
gives a general method of constructing P-domains that are not P$v$MD's.

$\bullet$ Irreversibility of $\Rightarrow _{4}$. An example of an essential domain which is not a P-domain
 was constructed   in \cite[Heinzer (1981)]{He}.  Recently, in \cite[Fontana-Kabbaj (2004), Example 2.3]{FK}, the authors show the existence of $n$-dimensional essential domains which are not P-domains, for all $n \geq 2$.

$\bullet$ Irreversibility of $\Rightarrow _{5}$. 
Note that,  by $\Rightarrow _{10}$, a CIC domain is a $v$--domain and Nagata, solving with a counterexample a famous conjecture stated by Krull in 1936,  has produced an example of a
one dimensional quasilocal CIC domain that is not a valuation ring (cf. \cite[Nagata (1952)]{Na1},  \cite[Nagata (1955)]{Na2}, and \cite[Ribenboim (1956)]{Rib}). This
proves that a $v$--domain may not be essential.  It would be  desirable  to have an
example of a nonessential $v$--domain that is simpler than Nagata's example.

$\bullet$ Irreversibility of $\Rightarrow _{6}$   (even in the Noetherian case).    This case can be
handled in the same manner as that of $\Rightarrow _{1}$, since a polynomial domain over a GCD domain is still a GCD domain (cf. \cite[Kaplansky (1970), Exercise 9, page 42]{Kap}). 

$\bullet$ Irreversibility of $\Rightarrow _{7}$   (even in the Noetherian case).   Note that a
Pr\"ufer domain  is a GGCD domain, since a GGCD domain is characterized by the fact that $F^v$ is invertible for all $F \in \f(D)$ \cite[Anderson-Anderson (1979), Theorem 1]{AA}. Moreover,  a Pr\"ufer domain 
$D$ is a B\'ezout domain if and only if $D$ is GCD. In fact,  according to  \cite[Cohn (1968)]{Co} a
Pr\"ufer domain $D$ is B\'ezout if and only if $D$ is a generalization of GCD domains
called a \emph{Schreier domain} (i.e., an integrally closed integral domain whose
group of divisibility is a \emph{Riesz group}, that is  a partially ordered directed group $G$ having the following interpolation property: given $a_1, a_2,..., a_m, b_1, b_2, ..., b_n \in G$ with $a_i \leq b_j, $ there exists $c\in G$ with $a_i \leq c\leq b_j$  see \cite[Cohn (1968)]{Co} and also \cite[D.D. Anderson (2000), Section 3]{Anderson-00}).  Therefore, a Pr\"ufer non-B\'ezout domain (e.g., a Dedekind non principal ideal domain, like $\mathbb Z[i\sqrt{5}]$) shows the irreversibility of $\Rightarrow _{7}$.

$\bullet$ Irreversibility of $\Rightarrow _{8}$.  
From the characterization of GGCD domains recalled in the irreversibility of $\Rightarrow _{7}$  \cite[Anderson-Anderson (1979), Theorem 1]{AA}, it follows that a GGCD domain is a P$v$MD.  More precisely, as we have already observed  just before Proposition \ref{pr:8}, an integral domain  
 $D$ is a GGCD domain if and only if  $D$ is a P$v$MD that is a
locally GCD domain.    Finally, as noted above, there are examples in \cite[Zafrullah (1988)]{Z2} of locally GCD domains that are not P$v$MD's. More explicitly,  let   $\boldsymbol E$ be the \emph{ring of entire
functions} (i.e.,  complex functions that are analytic in the whole plane).  It is well known that $\boldsymbol E$  is a B\'ezout domain and every nonzero non unit $x \in \boldsymbol E $ is uniquely expressible  as an associate of a  ``countable'' product 
$x =  \prod p_i^{e_i}$, 
where $e_i \geq 0$ and $p_i$  is an irreducible function (i.e., a function having a unique root) \cite[Helmer (1940), Theorems 6 and 9]{Helmer}.
 Let $S$ be the multiplicative set of $\boldsymbol E $ generated by the irreducible functions and let $X$ be an indeterninate over $\boldsymbol E $, then $\boldsymbol E  + X{\boldsymbol E}_S[X]$  is a locally GCD domain  that is not a P$v$MD \cite[Zafrullah (1988), Example 2.6 and Proposition 4.1]{Z2}.
 

$\bullet$ Irreversibility of $\Rightarrow _{9}$  (even in the Noetherian case).   This follows easily from the fact 
 that there do exist examples of Krull domains (which we have already observed are locally P$v$MD's)
that are not locally factorial   (e.g., a non-UFD local Noetherian integrally closed domain, like the power series domain $D[\![X]\!]$ constructed in   \cite[Samuel (1961)]{Samuel},  where $D$ is a two dimensional local Noetherian UFD).    As a matter of fact,  a Krull domain which is   a GCD domain      is a UFD, since in a GCD domain, for all $F \in \f(D)$, $F^v$ is principal and so the class group $\Cl(D) = 0$ \cite[Bouvier-Zafrullah (1988), Section 2]{BZ}; on the other hand,  a Krull domain is factorial if and only if $\Cl(D) = 0$ \cite[Fossum (1973), Proposition 6.1]{Fossum}.

$\bullet$ Irreversibility of $\Rightarrow _{10}$.  Let $R$ be an integral domain with quotient field $L$ and let $X$ be an indeterminate over $L$.  By  \cite[Costa-Mott-Zafrullah (1978), Theorem 4.42]{CMZ} $T:=
R+XL[X]$ is a $v$--domain if and only if  $R$ is a $v$--domain. Therefore, if $R$ is not equal to $L$, then obviously $T$ is an example of a $v$--domain
that is not completely integrally closed (the complete integral closure of $T$ is $L[X]$ \cite[Gilmer (1972), Lemma 26.5]{Gi}). This establishes that $\Rightarrow
_{10}$ is not reversible. \\
Note that,  in \cite[Fontana-Gabelli (1996), Section 4]{FoGa} the transfer in pullback diagrams of the P$v$MD property and related properties  is studied.  A characterization of $v$--domains in pullbacks is proved in   \cite[Gabelli-Houston (1997), Theorem 4.15]{GH}.  We summarize these results in the following:

 \begin{theorem}  Let $R$ be an integral domain with quotient field $k$ and let $T$ be an integral domain with a maximal ideal $M$  such that $L:=T/M$ is a field extension of $k$.   Let $ \varphi: T \rightarrow L$ be the canonical projection and consider the following pullback diagram:
 $$
\begin{CD}
D := \varphi^{-1}(R) @> >> R\\
@VVV @ VVV\\
T_1 := \varphi^{-1}(k) @>>> k \\
@VVV @ VVV\\
T@>\varphi >> L
\end{CD}
$$
 \\
 Then,  $D$ is a $v$--domain (respectively, a P$v$MD) if and only if  $k=L$, $T_M$ is a valuation domain and $R$ and $T$ are   $v$--domains (respectively,  P$v$MD's).
 \end{theorem}

 \begin{remark}
 \rm
  Recently, bringing to a sort of  close a lot of
efforts to restate results of \cite[Costa-Mott-Zafrullah (1978)]{CMZ} in terms of very general pullbacks,
 in the paper \cite[Houston-Taylor (2007)]{HT},  the authors use some remarkable techniques to prove   a generalization of   the previous theorem.
 Although   that   paper is not about $v$--domains in particular, 
it does have a few good results on $v$--domains.  One of these results will be recalled   in
 Proposition \ref{Proposition HT2}. Another one, with a pullback flavor, can be stated as follows:
\it Let $I$ be a nonzero ideal of an integral domain $D$ and set $T:=(I:I)$. If $D$ is a $v$--domain (respectively, a P$v$MD) then $T$ is a $v$--domain (respectively, a P$v$MD) \rm \cite[Houston-Taylor (2007), Proposition 2.5]{HT}.
\end{remark}

$\bullet$ Irreversibility of $\Rightarrow _{11}$.
Recall that an integral domain 
$D$ is called a \emph{Mori domain} if $D$ satisfies ACC on its integral
divisorial ideals. According to \cite[Nishimura (1963), Lemma 1]{Nishimura} or   \cite[Querr\'e (1971)]{Q},  $D$ is a Mori domain if and only if for
every nonzero integral ideal $I$ of $\ D$ there is a finitely generated
ideal $J\subseteq I$ such that $J^{v}=I^{v}$ (see also \cite[Barucci (2000)]{Barucci} for an updated survey on Mori domains). Thus,  if $D$ is a Mori domain
then $D$ is CIC (i.e., every nonzero ideal is $v$--invertible) if and only if $D$ is a $v$--domain (i.e., every nonzero finitely generated ideal is $v$--invertible).  On the other hand,  a completely integrally closed
Mori domain is a Krull domain (see for example   \cite[Fossum (1973), Theorem 3.6]{Fossum}).    More precisely,  Mori $v$--domains coincide with Krull domains \cite[Nishimura (1967), Theorem]{Nishi}.   Therefore an integrally closed Mori non Krull domain provides an example of the irreversibility of $\Rightarrow _{11}$.  An explicit example is given next.

It can be shown that, if $k \subseteq L$ is an extension of fields and
if $X$ is an indeterminate over $L$, then $k+XL[X]$ is always a Mori domain (see, for example,
\cite[Gabelli-Houston (1997), Theorem 4.18]{GH}  and references there  to previous papers by V. Barucci and M. Roitman).   It is
easy to see that the complete integral closure of $k+XL[X]$ is precisely  $L[X]$ \cite[Gilmer (1972), Lemma 26.5] {Gi}. Thus
if $k\subsetneqq L$ then $k+XL[X]$ is not completely integrally closed and, as an easy consequence of the definition of
integrality, it is integrally closed if and only if $k$ is algebraically closed in $L$. This shows that
there do exist integrally closed Mori domains that are not Krull. A very explicit example is given by $\overline{\mathbb Q}
+X\mathbb R[X]$,  where $\mathbb R$ is the field of real numbers and $\overline{\mathbb Q}$ is the
algebraic closure of $\mathbb Q$ in $\mathbb R$. 
\end{section}

\begin{section}{$v$--domains and rings of fractions}\label{fractions}

We have already mentioned that,    if $S$ is a multiplicative set of   a P$v$MD $D$,  then $D_{S}$ is still a P$v$MD  \cite[Heinzer-Ohm (1973), Proposition 1.8]{HO}.   The easiest proof of this fact can be given noting that, given $F \in \f(D)$, if $F$ is   $t$--invertible in $D$ then $FD_S$ is $t$--invertible in $D_S$,   where $S$ is a multiplicative set of $D$ \cite[Bouvier-Zafrullah (1988), Lemma 2.6]{BZ}.
 It is natural to ask  if  $D_{S}$ is a $v$--domain when $D$ is a $v$--domain.

The answer is no.
 As a matter of fact an example of an essential domain $D$
with a prime ideal $P$ such that $D_{P}$ is not essential was given in  \cite[Heinzer (1981)]{He}. What is
interesting is that an essential domain is a $v$--domain by Proposition \ref{ess-vdomain}
 and that, in this example, $D_{P}$ is a (non essential) overring of the type $
k+XL[X]_{(X)}=(k+XL[X])_{XL[X]} $, where $L$ is a field and $k$ its subfield
that is algebraically closed in $L$.
 Now, a domain of type  $k+XL[X]_{(X)}$ is an integrally
closed local Mori domain,  see \cite[Gabelli-Houston (1997), Theorem 4.18]{GH}.   In the irreversibility of $\Rightarrow _{11}$, we have also
observed that if a Mori domain   is a $v$--domain then it must be CIC, i.e., a Krull
domain, and hence, in particular,  an essential domain.  Therefore, Heinzer's construction provides an example of an essential ($v$--)domain  $D$ with a prime ideal $P$ such that $D_{P}$ is not a $v$--domain.

Note that a similar situation holds for CIC domains. If $D$ is CIC then it may be
that for some multiplicative set $S$ of $D$ the ring of fractions $D_{S}$ is not a completely integrally
closed domain.  A well known example in this connection is the ring $\boldsymbol E$ of entire
functions. For $ \boldsymbol E$ is a completely integrally closed B\'ezout domain that is
infinite dimensional (see [61 and 62, Henriksen (1952), (1953)],
 \cite[Gilmer (1972), Examples 16-21, pages 146-148]
{Gi} and  \cite[Fontana-Huckaba-Papick (1997), Section 8.1]{FHP}). 
Localizing $\boldsymbol E$ at one of its prime ideals of height greater than
one would give a valuation domain of dimension greater than one, which is
obviously not completely integrally closed \cite[Gilmer (1972), Theorem 17.5]{Gi}.   For another example of a CIC domain that has non--CIC rings of
fractions, look at the integral domain of integer-valued polynomials Int$(\mathbb Z)$ \cite[Anderson-Anderson-Zafrullah (1991), Example 7.7 and the following paragraph at page 127]{AAZ}. (This is a non-B\'ezout Pr\"ufer domain, being
atomic and two-dimensional.)   

Note that these examples,   like other well known examples of
CIC domains with some overring of fractions not CIC, are all such that their
 overrings of fractions are at least $v$--domains (hence, they do not provide further counterexamples to the transfer of the $v$--domain property to the overrings of fractions).   As a matter of fact, the examples that we have in mind are CIC B\'ezout  domains with Krull dimension $\geq 2$ (and polynomial domains over them), constructed using the Krull-Jaffard-Ohm-Heinzer Theorem (for the statement,  a brief history and applications of this theorem see \cite[Gilmer (1972), Theorem 18.6, page 214, page 136, Example 19.12]{Gi}).   Therefore,   it would be instructive to  find an example of a CIC domain whose  overrings of fractions are not all $v$--domains. Slightly more generally, we have the following.  
 
\smallskip
 
 It is well known that if $\{ D_{\lambda} \mid \lambda \in \Lambda \}$  is a family of overrings of $D$ with  
$D=\bigcap_{\lambda \in \Lambda}D_{\lambda }$ and if each $D_{\lambda}$ 
  is a completely integrally closed  (respectively,  integrally closed) domain then so is $D$    (for the completely integrally closed case see for instance  \cite[Gilmer (1972), Exercise 11, page
145]{Gi}; the integrally closed case is a straightforward consequence of the definition).   It is natural to ask if in the above statement ``completely integrally closed/integrally closed
domain''  is replaced by ``$v$--domain'' the statement is still true.

The answer in general is no, because  by Krull's theorem  every integrally closed
integral domain is expressible as an intersection of a family of its
valuation overrings (see e.g. \cite[Gilmer (1972), Theorem 19.8]{Gi}) and of course a
valuation domain is a $v$--domain. But, an integrally closed domain is not necessarily a $v$--domain (see the irreversibility of $\Rightarrow _{11}$).  If however each of $D_{\lambda}$ is a ring
of fractions of $D$, then the answer is yes.   A slightly more general statement is given next.

\begin{proposition}
\label{int-vdomain} Let $\{D_{\lambda } \mid \lambda\in \Lambda\}$ be a family of flat
overrings  of $D$ such that $D=\bigcap_{\lambda\in\Lambda} D_{\lambda }.$ If each of $
D_{\lambda }$ is a $v$--domain then so is $D$.
\end{proposition}

\vskip -0.3cm\begin{proof}  Let $v_\lambda$ be the $v$--operation on $D_\lambda$ and
let $\boldsymbol{\ast} := \wedge v_\lambda $, be the star operation on $D$ defined by $A\mapsto A^{\boldsymbol{\ast
}}: =\bigcap_\lambda (AD_{\lambda })^{v_\lambda}$, for all $A \in \F(D)$ \cite[D.D. Anderson (1988), Theorem 2]{A}.  To show that $D$ is a $v$--domain it is
sufficient to show that every nonzero finitely generated ideal is $\boldsymbol{\ast}$--invertible (for $\boldsymbol{\ast}  \leq v$ and so, if $F\in \f(D)$ and $(FF^{-1})^{\boldsymbol{\ast} }=D$, then applying
the $v$--operation to both sides we get $(FF^{-1})^{v}=D$).

 Now, we have 
  $$
 \begin{array}{rl}
(FF^{-1})^{\boldsymbol{\ast}  }= &  \bigcap_\lambda ((FF^{-1})D_{\lambda})^{v_\lambda}= \bigcap_\lambda
((FD_{\lambda})(F^{-1}D_{\lambda}))^{v_\lambda} \\
= &    \bigcap_\lambda
((FD_{\lambda})(FD_{\lambda})^{-1})^{v_\lambda} \mbox{ \, (since $D_{\lambda }$ is    $D$--flat  and $F$ is f.g.)} \\
= &    \bigcap_\lambda D_{\lambda}  \mbox{ \; (since  $D_{\lambda }$ is a $v_{\lambda }$--domain) \; } \\
= &   D \,.
\end{array}
$$ \vskip -15pt \end{proof}

\begin{corollary}
\label{Corollary 3} Let $\Delta$ be a nonempty family of prime ideals of $D$ 
such that $D=\bigcap \{D_{P} \mid P \in \Delta\}$. If $D_{P }$
is a $v$--domain  for each $P \in \Delta$, then  $D$ is a $v$--domain.  In particular,  if $D_M$ is a $v$--domain for all $M \in \Max(D)$ (for example, if $D$ is locally  a $v$--domain, i.e., $D_P$ is a $v$--domain for all $P \in \Spec(D)$), then $D$ is a $v$--domain.
\end{corollary}

  Note that the previous Proposition \ref{int-vdomain} and Corollary \ref{Corollary 3} generalize Proposition \ref{ess-vdomain}, which ensures that an essential domain is a $v$-domain.   Corollary \ref{Corollary 3}   in turn leads to an interesting conclusion concerning the overrings of fractions of a $v$--domain.

\begin{corollary}
\label{Corollary 4} Let $S$ be a multiplicative set in $D$. If $D_{P}$ is a $v$--domain for all prime
ideals $P$ of $D$ such that $P$ is maximal with respect to being disjoint from $S$,  
 then $D_{S}$ is a $v$--domain.
\end{corollary}

In  Corollary \ref
{Corollary 3}  we have shown that, if
$D_{M}$ is a $v$--domain for all $M \in \Max(D)$, then $D$ is a $v$--domain. However,  if
$D_{P}$ is a $v$--domain for all $P \in \Spec(D)$, we get much more in return. 
To indicate this,  we note that,
   if $S$ is a multiplicative set of $D$, then $D_{S}= \bigcap \{D_{Q} \mid 
 Q $  ranges over associated primes of principal ideals  of $D$ with $Q\cap
S=\emptyset   \}$     
\cite[Brewer-Heinzer (1974), Proposition 4]{BH}  (the definition of associated primes of principal ideals was recalled in  Remark \ref{rk:3}).  Indeed, if we let $S=\{1\}$, then we have $D=\bigcap D_{Q}\mid 
Q$ ranges over all associated primes of principal ideals of $D \}$ (see also \cite[Kaplansky (1970), Theorem 53]{Kap} for a ``maximal-type'' version of this property). Using this
terminology and the information at hand, it is easy to prove the following
result.

\begin{proposition}
\label{Proposition 5} Let $D$ be an integral domain. Then,  the following are
equivalent. 
\begin{itemize}
\item[\rm (i)] \hskip 10pt $D$ is a $v$-domain such that, for every multiplicative set 
$S$ of $D$,  $D_{S}$ is a $v$--domain.
\item[\rm (ii)] \hskip 10pt For every nonzero prime ideal $P$ of $D$, $
D_{P}$ is a $v$-domain.
\item[\rm (iii)]  \hskip 10pt For every associated prime of principal ideals  of $D$,  $Q$, $D_{Q}$
is a $v$--domain.
\end{itemize}
\end{proposition}

\smallskip

From the previous considerations, we have the following addition to the existing picture: 
$$
\begin{array}{rl}
\mbox{locally P$v$MD} \; \Rightarrow _{12} &  \mbox{locally $v$--domain} \; \Rightarrow _{13}  \mbox{$v$--domain.}
\end{array}
$$
The example discussed at the beginning of this section shows the irreversibility of  $\Rightarrow _{13}$.   Nagata's example (given for the irreversibility of $\Rightarrow _{5}$)  
of  a one dimensional quasilocal CIC domain that is not a valuation ring  shows also the irreversibility of  $\Rightarrow _{12}$.

\begin{remark} \rm
In the  spirit of Proposition \ref{Proposition 5},  we can make the following statement for CIC
domains: \it  Let $D$ be an integral domain. Then,  the following are
equivalent: 
\begin{itemize}
\item[\rm (i)]  \hskip 10pt $D$ is a CIC domain such that, for every multiplicative
set $S$ of $D$,  $D_{S}$ is CIC.

\item[\rm (ii)]  \hskip 10pt For every nonzero prime ideal $P$ of $D$, $
D_{P} $ is CIC.

\item[\rm (iii)]  \hskip 10pt For every associated prime of a principal ideal  of $D$,  $Q$, $D_{Q}$ is CIC.
\end{itemize} \rm

 At the beginning of this section, we have mentioned the existence of examples of $v$--domains (respectively, CIC domains) having some localization at prime ideals which is not a $v$--domain (respectively, a CIC domain). Therefore, the previous  equivalent properties (like  the equivalent properties of  Proposition \ref{Proposition 5}) are strictly stronger than the property of being   a  CIC domain  (respectively,  $v$--domain).
 
   On the other hand, for the case of   integrally closed domains, the fact that, for every nonzero prime ideal $P$ of $D$, $
D_{P} $ is integrally closed   (or, for every maximal ideal $M$ of $D$, $
D_{M} $ is integrally closed) returns exactly the property that $D$ is integrally closed (i.e., the ``integrally closed property'' is a local property; see, for example, \cite[Atiyah-Macdonald (1969), Proposition 5.13]{AM}).   Note that, more generally, the semistar integral closure is a local property (see for instance \cite[Halter-Koch (2003), Theorem 4.11]{H-K-03}). 
\end{remark}

\medskip

 We
have just observed that a ring of fractions of a $v$--domain may not be a $v$--domain, however there are distinguished classes of overrings for which the ascent of the $v$--domain property is possible.

 Given an extension of integral domains $D \subseteq T$ with the same field of quotients,  $T$ is called \emph{$v$--linked}
(respectively,  \emph{$t$--linked}) \emph{over $D$}  if whenever $I$ is a nonzero
(respectively, finitely generated) ideal of $D$ with $I^{-1}=D$ we have $
(IT)^{-1}=T$.    It is clear that
 $v$--linked implies   $t$--linked and it is not hard to prove that flat overring implies $t$--linked \cite[Dobbs-Houston-Lucas-Zafrullah (1989), Proposition 2.2]{DHLZ}.   Moreover, the complete integral closure and the pseudo-integral closure of an  integral domain $D$ are $t$--linked over $D$  (see \cite[Dobbs-Houston-Lucas-Zafrullah (1989), Proposition 2.2 and Corollary 2.3]{DHLZ} or \cite[Halter-Koch (1997), Corollary 2]{H-K1}).    Examples of $v$--linked extensions can be constructed as follows: take any nonzero ideal $I$ of an integral domain  then the overring $T := (I^v :I^v)$ is a $v$--linked overring of $D$ \cite[Houston-Taylor (2007), Lemma 3.3]{HT}.  
 
 The $t$--linked extensions were used in \cite[Dobbs-Houston-Lucas-Zafrullah (1989)]{DHLZ} to deepen the study of P$v$MD's. 
It is known that \it  an integral domain $D$ is a P$v$MD  if and only if each $t$--linked overring   of $D$   is a P$v$MD \rm  (see \cite[Houston (1986), Proposition 1.6]{Houston-86},  \cite[Kang (1989), Theorem 3.8 and Corollary 3.9]{Kang}). More generally, in  \cite[Dobbs-Houston-Lucas-Zafrullah (1989), Theorem 2.10]{DHLZ}, the authors  prove that \it  an integral domain $D$ is a P$v$MD if and only if each $t$--linked overring is integrally closed. \rm
On the other hand, a ring of fractions of a $v$--domain may not be a $v$--domain,
  so a $t$--linked overring of a $v$--domain
may not be a $v$--domain. However, when it comes to a $v$--linked overring we get a
different story. The following result is proven in \cite[Houston-Taylor (2007), Lemma 2.4]{HT}.

\begin{proposition}\label{Proposition HT2}
 If $D$ is a $v$--domain and $T$
is a $v$--linked overring of $D,$ then $T$ is a $v$-domain. \end{proposition}

\vskip -0.3cm \begin{proof} Let $J:= y_1T + y_2T + ... +y_nT$ be a nonzero finitely generated ideal of $T$ and set $F := y_1D + y_2D + ... +y_nD \in \f(D)$. Since $D$ is a $v$--domain, $(FF^{-1})^v = D$ and, since $T$ is $v$--linked, we have $(JF^{-1}T)^v =(FF^{-1}T)^v = T$. We conclude easily that $(J(T:J))^v =T$.
\end{proof}

\end{section}

\begin{section}{$v$--domains and polynomial extensions} \label{poly}

\noindent \bf 4.a The polynomial ring over a $v$--domain. \rm As for the case of integrally closed domains and  of completely integrally closed domains   \cite[Gilmer (1972), Corollary 10.8 and Theorem 13.9]{Gi},   we have observed in the proof of irreversibility of $
\Rightarrow _{1}$  that, \it  given an integral domain $D$ and an indeterminate $X$ over $D$,   
$$D[X]  \mbox{ is a P$v$MD} \; \Leftrightarrow  \; D   \mbox{ is a P$v$MD. }
$$  
\rm A similar statement holds for $v$--domains. As a matter of fact, \emph{the following statements are equivalent} (see part (4) of  \cite[D.D. Anderson-Kwak-Zafrullah (1995), Corollary 3.3]{AKZ}).
\it
\begin{itemize}
\item[\rm (i)] \hskip 10pt  For every $F \in \f(D)$, $F^v$ is $v$--invertible in $D$.
\item[\rm(ii)] \hskip 10pt For every $G \in \f(D[X])$, $G^v$ is $v$--invertible in $D[X]$.
\end{itemize}
\rm

This equi\-valence is essentially based on a polynomial characterization of integrally closed domains given in \cite[Querr\'e (1980)]{Querre}, for which we need some   introduction.   Given an integral domain $D$ with quotient field $K$, an indeterminate $X$ over $K$ and a polynomial $f \in K[X]$, we denote by $\co_D(f)$ the \emph{content of $f$}, i.e., the (fractional) ideal  of $D$ generated by the coefficients of $
f$. For every fractional ideal $B$ of $D[X]$,   set $\co_D(B) := (\co_D(f) \mid f \in B)$. 
The integrally closed domains are characterized by the following property: for each integral ideal $J$ of $D[X]$ such that $J \cap D \neq (0)$, $J^v = (\co_D(J)[X])^v = \co_D(J)^v[X]$  (see  \cite[Querr\'e (1980), Section 3]{Querre} and \cite[D.D. Anderson-Kwak-Zafrullah (1995), Theorem 3.1]{AKZ}). 
 Moreover,  an integrally closed domain is an \emph{agreeable domain} (i.e., for each fractional ideal $B$ of $D[X]$, with $B \subseteq K[X]$, there exists $0 \neq s \in D$ --depending on $B$-- with $sB \subseteq D$) \cite[D.D. Anderson-Kwak-Zafrullah (1995), Theorem 2.2]{AKZ}.    (Note that
 agreeable domains were also studied in \cite[Hamann-Houston-Johnson (1988)]{HHJ}  under the name of  almost principal ideal domains.)

 The previous considerations show  that,  for   an integrally closed domain $D$, there is a    close relation   between the divisorial ideals    of $D[X]$ and those of $D$  \cite[Querr\'e (1980), Th\'eor\`eme 1 and Remarque 1]{Querre}.
 The equivalence (i)$\Leftrightarrow$(ii) will now follow easily from the fact that, given an agreeable domain, for every integral ideal $J$ of $D[X]$, there exist an integral ideal $J_1$ of $D[X]$ with $J_1 \cap D \neq (0)$, a nonzero element $d \in D$ and a polynomial $f \in D[X]$ in such a way that $J = d^{-1}fJ_1$ \cite[D.D. Anderson-Kwak-Zafrullah (1995), Theorem 2.1]{AKZ}.

On the other hand, using the definitions of $v$--invertibility and   $v $--multiplica\-ti\-on, one can easily show that for $A\in \F(D),$ $A$ is $v$--invertible if and only if $A^{v}$ is $v$--invertible.   By the previous equivalence (i)$\Leftrightarrow$(ii),  we conclude that every $F\in \f(D)$ is $v$--invertible if and only if every $
G\in \f(D[X])$ is $v$--invertible and this proves the following:

\begin{theorem}  Given an integral domain $D$ and an indeterminate $X$ over $D$,  $D$ is a $v$-domain if and only if $
D[X]$ is a $v$--domain. 
\end{theorem}

Note that a much more interesting and general result was proved in
terms of pseudo-integral closures in \cite[D.F. Anderson-Houston-Zafrullah (1991), Theorem 1.5 and Corollary 1.6]{AHZ}.

\bigskip

\noindent \bf 4.b $v$--domains and rational functions. \rm Characterizations of $v$--domains can be also given in terms of rational functions,  using  properties of the content of polynomials.

Recall that Gauss' Lemma for the content of polynomials holds for Dedekind domains (or,
more generally, for Pr\"{u}fer domains).  A more precise and general
statement is given next.
 
\begin{lemma} \label{gauss} Let $D$  be an integral domain with quotient field $K$ and let   $X$ be an indeterminate over $D$. The following are equivalent. 
\begin{itemize}
\item[\rm (i)] \hskip 10pt $D$ is an integrally closed domain (respectively,  a P$v$MD; \ a Pr\"ufer domain).
\item[\rm (ii)] \hskip 10pt for all nonzero $f,g \in K[X]$, $\co_D(fg)^v = (\co_D(f)\co_D(g))^v$ (respectively, $\co_D(fg)^w  $ $= (\co_D(f)\co_D(g))^w$; \  $\co_D(fg) = \co_D(f)\co_D(g)$).
\end{itemize}
\end{lemma}
\noindent For the ``Pr\"ufer domain part'' of the previous lemma, see \cite[Gilmer (1972), Corollary 28.5]{Gi},  \cite[Tsang (1965)]{Tsang}, and \cite[Gilmer (1967)]{Gi-67};   for the ``integrally closed domains part'', see \cite[Krull (1936), page 557]{Krull:1936}  and \cite[Querr\'e (1980), Lemme 1]{Querre}; for the ``P$v$MD's part'', see \cite[D.F. Anderson-Fontana-Zafrullah (2008), Corollary 1.6]{AFZ}   and \cite[Chang (2008), Corollary 3.8]{C-08}. 
 For more on the history of Gauss' Lemma, the reader may consult   \cite[Heinzer-Huneke (1998), page 1306]{H-H}    and  \cite[D.D. Anderson (2000), Section 8]{Anderson-00}.

For general integral domains, we always have the inclusion of
ideals $ \boldsymbol{c}_{D}(fg)\subseteq
\boldsymbol{c}_{D}(f)\boldsymbol{c}_{D}(g)$, and, more precisely, we
have the following famous lemma due to Dedekind and Mertens (for the proof, see   \cite[Northcott (1959)]{Northcott}   or \cite[Gilmer (1972), Theorem
28.1]{Gi} and, for some complementary information, see \cite[D.D. Anderson (2000), Section 8]{Anderson-00}): 

\begin{lemma}  In the situation of Lemma \ref{gauss}, let $ 0 \ne f,g \in K[X]$ and let $m:=\deg (g)$.
Then
\begin{equation*}
\boldsymbol{c}_{D}(f)^{m}\boldsymbol{c}_{D}(fg)=\boldsymbol{c}_{D}(f)^{m+1}
\boldsymbol{c}_{D}(g)\,.
\end{equation*} \rm
\end{lemma}
\smallskip

A straightforward consequence of the previous lemma is the following:

\begin{corollary} \label{v-inv} In the situation of Lemma \ref{gauss}, assume that, for a nonzero polynomial $f \in K[X]$, $\co_D(f)$ is $v$--invertible (e.g., $t$--invertible). Then \  $\co_D(fg)^v $ $= (\co_D(f)\co_D(g))^v$ (or, equivalently, $\co_D(fg)^t = (\co_D(f)\co_D(g))^t$), for all nonzero $g \in K[X]$.
\end{corollary}

From Corollary \ref{v-inv} and from the ``integrally closed domain part'' of Lemma \ref{gauss}, we have the following result (see \cite[Mott-Nashier-Zafrullah (1990), Theorem 2.4 and Section 3]{MNZ}):

\begin{corollary}
In the situation of Lemma \ref{gauss}, set
${\tV}_{D} :=\{g\in D[X] \mid  \co_D(g) \mbox{ is $v$--invertible}\}\mbox{ and }{\tT}_{D} :=\{g\in D[X] \mid  \co_D(g) \mbox{ is $t$--invertible}\}.$
Then, $ \tT_D$ and ${V}_{D}$  are  multiplicative sets of $D[X]$ with $ \tT_D \subseteq {\tV}_{D}$. Furthermore,   ${\tV}_{D}$ (or, equivalently, $\tT_D$) is saturated if and only if $D$ is integrally closed.

\end{corollary}

It can be useful   to observe  that,  from Remark \ref{intersection-princ} (a.1), we have     
$$
\emph{\tT}_{D} = \{ g \in \emph{\tV}_D \mid  \co_D(g)^{-1} \mbox{ is $t$--finite}\}. 
$$
We are now   in a position    to give a characterization of $v$--domains (and P$v$MD's) in terms of rational functions    (see \cite[Mott-Nashier-Zafrullah (1990), Theorem 2.5 and Theorem 3.1]{MNZ}).   

\begin{theorem}\label{v-rational}
\label{Theorem MNZ} Suppose that $D$ is an
integrally closed domain, then the following are equivalent:
\begin{itemize}
\item[\rm (i)] \hskip 10pt $D$ is a $v$-domain (respectively, a P$v$MD).

\item[\rm (ii)] \hskip 10pt
$\tV_{D}=D[X]\backslash \{0\}$ (respectively,  $\tT_{D}=D[X]\backslash \{0\}$).

\item[\rm (iii)] \hskip 10pt$D[X]_{{\footnotesize{\tV}}_{D}}$
 (respectively, $D[X]_{{\footnotesize{\tT}}_{D}}$)
   is a field (or, equivalently, $D[X]_{{\footnotesize{\tV}}_{D}}= K(X)$
    (respectively, $D[X]_{{\footnotesize{\tT}}_{D}}= K(X)$)).

\item[\rm (iv)] \hskip 10ptEach
nonzero element $z \in K$ satisfies a polynomial $f\in D[X]$ such that $
\co_D(f)$ is $v$--invertible  (respectively, $t$--invertible).
\end{itemize}
\end{theorem}

\begin{remark} \rm Note that \emph{quasi Pr\"ufer domains} (i.e., integral domains having integral closure Pr\"ufer   \cite[Ayache-Cahen-Echi (1996)]{ACE})  can also be characterized by using properties of the field of rational functions. In the situation of Lemma \ref{gauss}, set $\emph{\tS}_D :=\{g\in D[X] \mid  \co_D(g) \mbox{ is invertible}\}$. Then, by Lemma \ref{v-inv}, the multiplicative set $\emph{\tS}_D$ of $D[X]$    is saturated    if and only if $D$ is integrally closed. Moreover, $D$ is quasi Pr\"ufer if and only if $D[X]_{{\footnotesize\emph{{\tS}}}_{D}}$ is a field (or, equivalently, $D[X]_{{\footnotesize\emph{{\tS}}}_{D}}= K(X)$) if and only if  each
nonzero element $z \in K$ satisfies a polynomial $f\in D[X]$ such that $
\co_D(f)$ is  invertible \cite[Mott-Nashier-Zafrullah (1990), Theorem 1.7]{MNZ}.
\end{remark}

Looking more carefully at the content of polynomials, it  is obvious that the set
$$\emph{\tN}_{D} :=\{g\in D[X] \mid  \co_D(g)^v = D \} $$
is a subset of $\emph{\tT}_D$ and it is well known that $\emph{\tN}_{D}$ is a saturated multiplicative set of $D[X]$ \cite[Kang (1989), Proposition 2.1]{Kang}.  We call the \emph{ Nagata ring of $D$ with respect to the $v$--operation} the ring:
$$ \Na(D, v) := D[X]_{\footnotesize\emph{{\tN}}_{D}}\,.$$
~We can  also consider
$$
\Kr(D, v) := \{ f/g \mid f, g \in D[X], \  g\neq 0, \ \co_D(f)^v \subseteq \co_D(g)^v \}\,.  $$
When $v$ is an e.a.b. operation on $D$ (i.e., when $D$ is a $v$--domain) $
\Kr(D, v)$ is a ring called  the \emph{Kronecker function ring of $D$ with respect to the $v$--operation} \cite[Gilmer (1972), Theorem 32.7]{Gi}.
Clearly, in general,  $ \Na(D, v)  \subseteq \Kr(D, v)$. It is proven in \cite[Fontana-Jara-Santos (2003), Theorem 3.1 and Remark 3.1]{FJS} that $ \Na(D, v)$ $  = \Kr(D, v)$ if and only if $D$ is a P$v$MD.

\begin{remark} (a) Concerning Nagata and Kronecker function rings, note that  a unified general treatment and semistar analogs of several results were obtained in the recent years, see for instance  \cite[Fontana-Loper (2001)]{FLo-01}, \cite[Fontana-Loper (2006)]{FLo-06} and   \cite[Fontana-Loper (2007)]{FLo-07}.

(b) A general version of Lemma \ref{gauss}, in case of semistar operations, was recently proved in \cite[D.F. Anderson-Fontana-Zafrullah (2008), Corollary 1.2]{AFZ}.
\end{remark}

\bigskip


 \noindent \bf 4.c $v$--domains and uppers to zero. \rm Recall that if $X$
is an indeterminate over an integral domain $D$ and if $Q$ is a  nonzero  prime ideal
of $D[X]$ such that $Q\cap D= (0)$ then $Q$ is called \emph{an upper to zero}. 
The   ``upper'' terminology    in polynomial rings is due to  S. McAdam   and was introduced in the early  1970's.  
In a recent paper, Houston and Zafrullah introduce the
\emph{UM$v$--domains} as the integral domains such that the uppers to zero are maximal $v$--ideals and they prove  the
following result      \cite[Houston-Zafrullah (2005), Theorem 3.3]{HZ1}.

\begin{theorem}
\label{umv} Let $D$ be an integral domain with quotient field $K$ and let $X$ be an indeterminate over $K$.
The following are equivalent.
\begin{itemize}
\item[\rm (i)] \hskip 10pt $D$ is a $v$-domain.
\item[\rm (ii)] \hskip 10pt $D$ is an integrally closed UM$v$--domain.
\item[\rm (iii)]  \hskip 10pt $D $ is integrally closed and every upper to zero in $D[X]$ is $v$--invertible.
\item[\rm (iii$_\ell$)] \hskip 10pt  $D$ is integrally closed and every upper to zero of the type $Q_{\ell}:=\ell K[X]\cap D[X]$ with $\ell \in D[X]$ a linear polynomial is $v$--invertible.
\end{itemize}
\end{theorem}

It would be unfair to end the section with this characterization of $v$-domains
without giving a hint about where the idea came from.

Gilmer  and  Hoffmann  in  1975  gave a characterization of Pr\"ufer domains using uppers to zero. This result is based on the following characterization of essential valuation overrings of    an integrally closed domain $D$: let  $P$ be a prime ideal of $D$, then $D_P$ is a valuation domain if and only if, for each upper to zero $Q$ of $D[X]$, $Q \not\subseteq P[X]$,  \cite[Gilmer (1972), Theorem 19.15]{Gi}.

%

A globalization of the previous statement leads to the following result that can be easily deduced from  \cite[Gilmer-Hoffmann (1975), Theorem 2]{Gi-Ho}.

\begin{proposition} \label{pr:26} In the situation of Theorem  \ref{umv}, the following are equivalent:
\begin{itemize}
\item[\rm (i)]  \hskip 10pt $D$ is a Pr\"ufer domain.
\item[\rm (ii)]  \hskip 10pt $D$ is integrally closed and if $Q$  is an upper to zero    of $D[X]$, then $Q \not\subseteq M[X]$, for all $M \in \Max(D)$ (i.e., $\co_D(Q) =D$).
\end{itemize}
\end{proposition}

In \cite[Zafrullah (1984), Proposition 4]{Z1} the author proves  a ``$t$--version'' of the previous result.

\begin{proposition}
\label{Theorem 8}
In the situation of Theorem  \ref{umv}, the following are equivalent:
\begin{itemize}
\item[\rm (i)]  \hskip 10pt $D$ is a P$v$MD.

\item[\rm (ii)] \hskip 10pt $D$ is integrally closed and  if $Q$ upper to zero  of $D[X]$, then $Q \not\subseteq M[X]$, for all maximal $t$--ideal $M$ of $D$ (i.e.,  $\co_D(Q)^t =D$).
\end{itemize}
\end{proposition}

The proof of the previous proposition relies on very basic properties of polynomial rings. 

Note that in \cite[Zafrullah (1984), Lemma 7]{Z1} it
is also shown that, if $D$ is a P$v$MD, then every upper to zero
in $D[X]$ is a maximal $t$--ideal. 
 As we observed in Section \ref{prel}, unlike maximal $v$--ideals, the maximal $t$--ideals are  ubiquitous.

Around the same time, in \cite
[Houston-Malik-Mott (1984), Proposition 2.6]{HMM}, the authors came up with a much better result,  using
the $\ast $-operations much more efficiently. Briefly, this result  said that the converse holds, i.e., \it 
$D$ is a P$v$MD if and only if  $D$ is  an integrally closed integral domain and every
upper to zero in $D[X]$ is a maximal $t$--ideal. \rm

It turns  out that integral
domains $D$ such that their 
uppers to zero in $D[X]$ are maximal $t$--ideals  (called  \emph{UM$t$-domains}  in   \cite[Houston-Zafrullah (1989), Section 3]{HZ2};  see also \cite[Fontana-Gabelli-Houston (1998)]{FGH} and, for a survey on the subject, \cite[Houston (2006)]{Houston})  and domains such that, for each upper to zero $Q$ of $D[X]$, $\co_D(Q)^t = D$
  had   an independent life.  In  \cite[Houston-Zafrullah (1989), Theorem 1.4]{HZ2},   studying $t$--invertibility, the authors  prove the following result.

\begin{proposition}
\label{Theorem 9}  In the situation of Theorem \ref{umv}, let $Q$ be an upper to zero in $D[X]$. The following statements are equivalent.
\begin{itemize}
\item[\rm (i)]  \hskip 10pt $Q$ is a maximal $t$--ideal of $D[X]$.
\item[\rm (ii)] \hskip 10pt $Q$ is a $t$--invertible ideal of $D[X]$.
\item[\rm (iii)] \hskip 10pt $\co_D(Q)^{t}=D$.
\end{itemize}
 \end{proposition}

Based on this result, one can see that the following statement  is  a  precursor to Theorem \ref{umv}.

\begin{proposition}
\label{Theorem 10}   Let $D$ be an integral domain with quotient field $K$ and let $X$ be an indeterminate over $K$. The following are equivalent.
\begin{itemize}
\item[\rm (i)] \hskip 10pt $D$ is a P$v$MD.
\item[\rm (ii)] \hskip 10pt $D$ is an integrally closed UM$t$--domain.
\item[\rm (iii)] \hskip 10pt $D$ is
integrally closed and every upper to zero in $D[X]$ is $t$--invertible.
 \item[\rm (iii$_\ell$)]  \hskip 10pt  $D $ is integrally closed and every upper to zero of the type  $Q_{\ell}:=\ell K[X]\cap D[X]$, with $\ell \in D[X]$ a linear polynomial, is $t$--invertible.
 \end{itemize}
\end{proposition}

Note that  the equivalence (i)$\Leftrightarrow$(ii) is  in  \cite[Houston-Zafrullah (1989), Proposition 3.2]{HZ2}.   (ii)$\Leftrightarrow$(iii) is a consequence of previous Proposition \ref{Theorem 9}. Obviously, (iii)$\Rightarrow$(iii$_\ell$). (iii$_\ell$)$\Rightarrow$(i) is a consequence of the characterization already cited that an integral domain  $D$ is a P$v$MD if and only if each nonzero two generated ideal is $t$--invertible \cite[Malik-Mott-Zafrullah (1988), Lemma 1.7]{MMZ}.   As a matter of fact,  consider a nonzero  two generated ideal $I:= (a, b)$ in $D$, set $\ell:= a + bX$ and  $Q_{\ell}:=\ell K[X]\cap D[X]$. Since $D$ is integrally closed, then $Q_{\ell} = \ell\co_D(\ell)^{-1}D[X]$   by \cite[Querr\'e (1980), Lemme 1, page 282]{Querre}. If $Q_{\ell}$ is $t$--invertible (in ($D[X]$), then it is easy to conclude that   $\co_D(\ell)= I$ is $t$--invertible (in $D$).

\medskip

 \begin{remark} \rm  Note that Pr\"ufer domains may not be characterized by straight modi\-fications of conditions (ii) and (iii) of Proposition \ref{Theorem 10}.
 As a matter of fact, if there exists in $D[X]$ an upper to zero which is also a maximal ideal, then the domain $D$ is a G(oldman)-domain (i.e., its quotient field is finitely generated over $D$), and conversely   \cite[Kaplansky (1970), Theorems 18 and 24]{Kap}.  Moreover, every upper to zero in $D[X]$ is invertible  if and only if $D$ is a GGCD domain  \cite[D.D. Anderson-Dumitrescu-Zafrullah (2007), Theorem 15]{ADZ}. 
 
 On the other  hand,   a variation of condition (iii$_\ell$) of Proposition \ref{Theorem 10} does characterize Pr\"ufer domains:  \it $D$ is a Pr\"ufer domain if and only if $D$ is integrally closed and  every upper to zero of the type $Q_{\ell}:=\ell K[X]\cap D[X]$ with $\ell \in D[X]$ a linear polynomial is such that $\co_D(Q_{\ell}) =D$ \rm \cite[Houston-Malik-Mott (1984), Theorem 1.1]{HMM}. 
 
 \end{remark}

\end{section}

\begin{section}{$v$--domains and GCD--theories}

In \cite[Borevich-Shafarevich (1966), page 170]{BS},  a \emph{factorial monoid} $\boldsymbol{\mathcal{D}}$   is a commutative semigroup with a unit element   $\mathfrak{1}$  (and without zero element) such that every element $\mathfrak{a} \in \boldsymbol{\mathcal{D}}$ can be uniquely represented as a finite product of atomic (= irreducible) elements $\mathfrak{q}_i$ of $\boldsymbol{\mathcal{D}}$, i.e., $ \mathfrak{a} = \mathfrak{q}_1\mathfrak{q}_2...\mathfrak{q}_r$, with $r \geq 0$ and this factorization is unique up to the order of factors; for $r=0$ this product is set equal to  $\mathfrak{1}$.    As a consequence, it is easy to see that this kind of uniqueness of factorization implies that  $\mathfrak{1}$  is the only invertible element in $\boldsymbol{\mathcal{D}}$, i.e., $\U(\boldsymbol{\mathcal{D}}) =   \{\mathfrak{1}\}$. Moreover, it is not hard to see that, in a factorial monoid, any two elements have GCD and every atom is a prime element \cite[Halter-Koch (1998), Theorem 10.7]{H-K}.

Let $D$ be an integral domain and set $D^\bullet := D \setminus\{0\}$. 
In \cite[Borevich-Shafarevich (1966), page 171]{BS} an integral domain $D$ is said to have a \emph{divisor theory}
if there is a factorial monoid $\boldsymbol{\mathcal{D}}$ and a semigroup homomorphism, denoted by 
(--)$: D^\bullet \rightarrow \boldsymbol{\mathcal{D}}$, given by $a\mapsto (a)$,  such
that:
\begin{itemize}

\item[\bf(D1)] \hskip 14pt $(a) \mid(b)$ in $\D$ if and only if $a \mid
b $ in $D$ for $a, b \in D^\bullet$.

\item[\bf(D2)]  \hskip 14pt   If $\mathfrak{g}\mid (a)$ and $\mathfrak{g}\mid (b)$ then $\mathfrak{g}\mid(a \pm b )$ for $a, b
\in D^\bullet$ with $a \pm b \neq 0$ and $\mathfrak{g}\in 
\D$.

\item[\bf(D3)] \hskip 14pt  Let $\mathfrak{g} \in \D$ and set
$$\overline{\mathfrak{g}}:=\{ x \in D^\bullet \mbox {\; such that \;}  \mathfrak{g}\mid (x)\}\cup \{0\}.
$$
Then $\overline{\mathfrak{a}}=\overline{\mathfrak{b}}$ if and only if $
\mathfrak{a=b}$ for all $\mathfrak{a,b } \in \D$.
\end{itemize}

 Given a divisor theory, the elements of the factorial monoid $\D$ are called \emph{divisors of the integral domain $D$} and the divisors of the type $(a)$, for $a \in D$ are called \emph{principal divisors of $D$}.

Note that, in \cite[Skula (1970), page 119]{Skula},  the author shows that the axiom {\bf{(D2)}}, which  guarantees that $\overline{\mathfrak{g}}$ is an ideal of $D$,  for each divisor $\mathfrak{g} \in \D$,  is unnecessary.   Furthermore, note that divisor theories were also considered in \cite[Mo\v{c}ko\v{r} (1993), Chapter 10]{Mck},  written in the spirit of Jaffard's volume \cite{J:1960}.

Borevich and Shafarevich   introduced domains with a divisor theory  in order to generalize  Dedekind domains and unique factorization domains, along the lines of Kronecker's  classical theory of ``algebraic divisors'' (cf.   \cite[Kronecker (1882)]{Kronecker} and also  \cite[Weyl (1940)]{Weyl} and \cite[Edwards [1990)]{Edwards}). As a matter of fact, they proved that
\begin{enumerate}
\item[(a)]  if an integral domain $D$ has a divisor theory  (--)$: D^\bullet \rightarrow \boldsymbol{\mathcal{D}}$  then it has only one (i.e., if (\!(--)\!)$: D^\bullet \rightarrow \boldsymbol{\mathcal{D'}}$ is another divisor theory then there is an isomorphism $\boldsymbol{\mathcal{D}} \cong   \boldsymbol{\mathcal{D'}}$  under which the principal divisors in $ \boldsymbol{\mathcal{D}}$  and $\boldsymbol{\mathcal{D'}}$, which correspond to a given nonzero element $a \in D$,  are identified) \cite[Borevich-Shafarevich (1966), Theorem 1, page 172]{BS};

\item[(b)]  an integral domain $D$ is a unique factorization domain if and only if  $D$  has a divisor theory  (--)$: D^\bullet \rightarrow \boldsymbol{\mathcal{D}}$ in which every divisor of $\D$ is principal \cite[Borevich-Shafarevich (1966), Theorem 2, page 174]{BS};

\item[(c)]  an integral domain $D$ is a Dedekind domain  if and only if  $D$  has a divisor theory  (--)$: D^\bullet \rightarrow \boldsymbol{\mathcal{D}}$ such that,  for every prime element $ {\mathfrak{p}}$ of $\D$, $D/\overline{\mathfrak{p}}$ is a field \cite[Borevich-Shafarevich (1966), Chapter 3, Section 6.2]{BS}.
\end{enumerate}

Note that Borevich and Shafarevich do not enter into the details of the determination of those integral domains for which a theory of divisors can be constructed \cite[Borevich-Shafarevich (1966), page 178]{BS}, but it is known that they coincide with the Krull domains (see \cite[van der Waerden (1931), \S 105]{vdW0}, \cite[Aubert (1983), Theorem 5]{Aubert},  \cite[Lucius (1998), \S 5]{Lu},  and \cite[Krause (1989)]{Krause} for the monoid case).    In particular, note that, for a Krull domain, the
group of non-zero fractional divisorial ideals provides a divisor theory.

\medskip

Taking the above definition as a starting point and recalling that \bf (D2) \rm  is
unnecessary, in  \cite[Lucius (1998)]{Lu}, the author  introduces a more general class of domains, called the domain with a GCD--theory.

 An integral domain $D$ is said to have a \emph{GCD--theory}
if there is a GCD--monoid $\G$ and a semigroup homomorphism, denoted by 
(--)$: D^\bullet \rightarrow \G$, given by $a\mapsto (a)$,  such
that:
\begin{itemize}

\item[\bf(G1)] \hskip 15pt $(a) \mid(b)$ in $\G$ if and only if $a \mid
b $ in $D$ for $a, b \in D^\bullet$.

\item[\bf(G2)] \hskip 15pt Let $\mathfrak{g} \in \G$ and set
$\overline{\mathfrak{g}}:=\{ x \in D^\bullet  \mbox{\, such that \,}  \mathfrak{g}\mid (x)\}\cup \{0\}.
$
Then $\overline{\mathfrak{a}}=\overline{\mathfrak{b}}$ if and only if $
\mathfrak{a=b}$ for all $\mathfrak{a,b } \in \G$.
\end{itemize}

Let $\Q := \boldsymbol{q}(\G)$ be the  group of quotients of the GCD--monoid $\G$. 
It is not hard to prove that the natural
extension a GCD--theory (--)$: D^\bullet \rightarrow \G$ to a
group homomorphism (--)$^{\prime }:  K^\bullet \rightarrow \Q$
 has the following properties:
 \begin{itemize}

\item[\bf(qG1)] \hskip 20pt $(\alpha)^{\prime }\mid (\beta)^{\prime }$ with respect to $\G
$ if and only if $\alpha \mid \beta $ with respect to $D$ for $\alpha, \beta $ $
\in K^\bullet.$

\item[\bf(qG2)] \hskip 20 pt Let $\mathfrak{h} \in \Q$ and set
$\overline{\mathfrak{h}}:=\{ \gamma \in K^\bullet \mbox{\, such that \,}  \mathfrak{h}\mid (\gamma)^\prime\}\cup \{0\}$ (the division in $\Q$ is with respect to $\G$).
Then $\overline{\mathfrak{a}}=\overline{\mathfrak{b}}$ if and only if $
\mathfrak{a=b}$ for  all $\mathfrak{a,b\in } \Q$.
\end{itemize}

In \cite[Lucius (1998), Theorem 2.5]{Lu},  the author proves the following key result, that clarifies the role
of the ideal $\overline{\mathfrak{a}}$.   (Call, as before, \emph{divisors of  $D$} the elements of the GCD--monoid $\G$   and  \emph{principal divisors of $D$} the divisors of the type $(a)$, for $a \in D^\bullet$.)

\begin{proposition}
\label{Theorem L1}  Let $D$ be an integral domain with
GCD--theory \ \rm (--)$: D^\bullet \rightarrow \G$,  \it let  $\mathfrak{a}$ be
any divisor of $\G$ and $\{(a _{i}) \}_{i\in I}$ a family of principal divisors
with $\mathfrak{a=}$ GCD$(\{(a_{i})\}_{i\in I})$. Then $\overline{\mathfrak{a}}=(\{a _{i}\}_{i\in I})^{v}=\overline{\mathfrak{a}}^v$.
\end{proposition}

Partly as a consequence of Proposition \ref{Theorem L1}, we have a
characterization of a $v$--domain as a domain with GCD-theory \cite[Lucius (1998), Theorem and Definition 2.9]{Lu}.

\begin{theorem}
\label{Theorem L2} Given an integral domain $D$, $D$ is a ring with GCD--theory if and only if $D$ is a $v$--domain.
\end{theorem}

 The ``only if part''  is a consequence of Proposition \ref{Theorem L1} (for details see
 \cite[Lucius (1998), Corollary 2.8] {Lu}). 
 
 The proof of the ``if part''  is constructive and provides explicitly the GCD--theory. The GCD--monoid is  constructed via  Kronecker function rings. Recall that, when $v$ is an e.a.b. operation (i.e.,  when $D$ is a $v$--domain (Theorem \ref{lorenzen})), the Kronecker function ring with respect to $v$, $\Kr(D,v)$, is well defined  and  is a B\'ezout domain \cite[Gilmer (1972), Lemma 32.6 and Theorem 32.7]{Gi}.    Let $\K$ be the monoid $\Kr(D, v)^\bullet$, let $\U := \U(\Kr(D, v))$ be the group of invertible elements in $\Kr(D, v)$ and set $\G := \K/\U$. The canonical map:
 $$
  [\mbox{--}]: D^\bullet \rightarrow \G=\frac{\Kr(D, v)^\bullet}{\U},\;\, a \mapsto [a] \ (=\mbox{the equivalence class of $a$ in } \G)
  $$
 defines a GCD--theory for $D$, called the  \emph{Kroneckerian GCD--theory} for the $v$--domain $D$.   In particular, the GCD of  elements in $D$ is realized by the equivalence class of a polynomial; more precisely, under this GCD--theory, given $a_0, a_1, ..., a_n \in D^\bullet$,  GCD$(a_0, a_1, ..., a_n) :=$ GCD$([a_0], [a_1], ..., [a_n])=  [a_0+a_1X+ ...+ a_nX^n]$.

  It is classically known \cite[Borevich-Shafarevich (1966), Chapter 3, Section 5]{BS} that the integral closure of a domain with divisor theory in a finite extension of fields is again a domain with divisor theory. For integral domains with GCD--theory a stronger result holds.

\begin{theorem}
\label{Theorem L3}  Let $D$ be an integrally closed
domain with field of fractions $K$ and let $K \subseteq L$ be an algebraic field extension and let $T$ be the integral closure of $D$
in $L$. Then $T$ is a $v$--domain (i.e., domain with GCD--theory) if and
only if $D$ is   a  $v$--domain (i.e.,  a    domain with GCD--theory).
\end{theorem}

The proof of the previous result is given in \cite[Lucius (1998), Theorem 3.1]{Lu} and it is based on the following facts:

 In the situation of Theorem \ref{Theorem L3},
\begin{itemize}
\item[(a)] \hskip 10pt For each ideal $I$ of $D$, $I^{v_D} = (IT)^{v_T} \cap K $ \cite[Krull (1936), Satz 9,  page 675]{Krull:1936};

\item[(b)]  \hskip 10pt If $D$ is a $v$--domain, then the integral closure of $\Kr(D, v_D)$ in the algebraic field extension $K(X) \subseteq L(X)$ coincides with $\Kr(T, v_T)$   \cite[Lucius (1998), Theorem 3.3]{Lu}. 
\end{itemize}

\begin{remark} \rm
 (a) The notions of   GCD--theory and divisor theory, being more in the setting of monoid theory, have been given a monoid treatment \cite[Halter-Koch (1998), Exercises 18.10, 19.6 and  Chapter 20]{H-K}.

 (b)  Note that a part of  previous Theorem \ref{Theorem L3} appears also as a corollary to \cite[Halter-Koch (2003),Theorem 3.6]{H-K2}. More precisely, let
$ \cl^v(D) \ (:= \bigcup \{F^v : F^v) \mid F \in \f(D)\})$ be the $v$--(integral) closure  of $D$. We have already observed (Theorem \ref{lorenzen} and Remark \ref{pseudo}) that an integral domain $D$ is a $v$--domain if and only if $D = \cl^v(D)$.
Therefore  Theorem \ref{Theorem L3} is an easy consequence of the fact that, in the situation of Theorem \ref{Theorem L3},  it can be shown that $\cl^{v}(T)$ is the integral closure of $\cl^{v}(D)$ in $L$ \cite[Halter-Koch (2003), Theorem 3.6]{H-K2}.

 (c)  In \cite[Lucius (1998), \S 4]{Lu}, the author develops a ``stronger GCD--theory'' in order to characterize P$v$MD's. A \emph{GCD-theory of finite type} is a GCD--theory, $(...)$,  with the property that each divisor $\frak{a}$ in the  GCD--monoid $\G$ is such that $\frak{a} = $ GCD$((a_1), (a_2),..., (a_n))$ for a finite number of nonzero elements $ a_1, a_2,..., a_n \in D$.  For a P$v$MD,
the group of non-zero fractional $t$--finite $t$--ideals provides a GCD--theory of finite type.
(Note that the notion  of a GCD--theory of finite type was introduced in \cite[Aubert (1983)]{Aubert} under the name of ``quasi  divisor theory''.   A thorough presentation of this concept, including several characterizations of P$\ast$MD's, is in \cite[Halter-Koch (1998), Chapter 20]{H-K}.)  

The  analogue  of Theorem \ref{Theorem L2} can be stated as follows: \it Given an integral domain $D$, $D$ is a ring with GCD--theory of finite type if and only if $D$ is a P$v$MD. \rm
  Also in this case, the GCD--theory of finite type and the GCD--monoid are constructed explicitly, via the Kronecker function ring $\Kr(D,v)$ (which coincides in this situation with the Nagata ring $\Na(D, v)$), for the details see \cite[Lucius (1998), Theorem 4.4]{Lu}.  Moreover, in \cite[Lucius (1998), Theorem 4.6]{Lu} there is given another proof   of  Pr\"ufer's theorem \cite[Pr\"ufer (1932), \S 11]{Prufer-32}, analogous to Theorem \ref{Theorem L3}:   \it    Let $D$ be an integrally closed
domain with field of fractions $K$ and let $K \subseteq L$ be an algebraic field extension and let $T$ be the integral closure of $D$
in $L$. Then $T$ is a P$v$MD (i.e., domain with GCD--theory of finite type) if and
only if $D$ is P$v$MD (i.e., domain with GCD--theory of finite type). \rm  Recall that a similar result holds for the special case of Pr\"ufer domains \cite[Gilmer (1972), Theorem 22.3]{Gi}.
\end{remark}

 \newpage
\section{Ideal-theoretic characterizations of $v$--domains}

Important  progress in the knowledge of the ideal theory for $v$--domains  was made   in 1989, after a series of talks given by   the  second named author while  visiting  seve\-ral   US universities.  The results of various discussions of that period are contai\-ned
in the
``A to Z" paper \cite[Anderson-Anderson-Costa-Dobbs-Mott-Zafrullah (1989)]{A-Z}, which contains in particular some new characterizations
of $v$-domains and of completely integrally closed domains. These
characterizations were then expanded into a very long list of   equivalent    statements, providing further characterizations of (several classes of) $v$--domains \cite[Anderson-Mott-Zafrullah (1989)]{AMZ}.

\begin{proposition} \label{inter} Let $D$ be an integral domain. Then, $D$ is a $v$--domain if and only if $D$ is integrally closed and $(F_1 \cap F_2 \cap ... \cap F_n)^v = F_1^v \cap F_2^v \cap ... \cap F_n^v$ for all $F_1, F_2,  ... ,  F_n \in \f(D)$ (i.e., the $v$--operation distributes over finite intersections of finitely generated fractional ideals).
\end{proposition}

The  ``if part''  is contained in the ``A to Z" paper (Theorem 7 of that paper,   where the converse was left open).   
The converse of this result was proved   a few  years later   in \cite[Matsuda-Okabe (1993), Theorem 2]{MO}.

Note that, even  for a Noetherian 1-dimensional domain, the $v$--operation may not distribute over finite intersections of (finitely generated) fractional ideals.  For instance,  here is an example due to W. Heinzer cited in \cite[D.D. Anderson-Clarke (2006), Example 1.2]{AC2},  
 let $k$ be a field, $X$ an indeterminate over $k$ and set $D:=k[\![X^3, X^4, X^5]\!]$, $F := (X^3, X^4)$ and $G:=(X^3, X^5)$. Clearly, $D$ is a non-integrally closed 1-dimensional local Noetherian domain with maximal ideal $M:= (X^3, X^4, X^5) = F+G$. It is easy to see that    $F^v =G^v = M$,    and so $F\cap G =(X^3) = (F\cap G)^v \subsetneq F^v \cap G^v = M$.

Recently, D.D. Anderson and Clarke have investigated the star operations that distribute over finite   intersections. In particular,   in \cite[D.D. Anderson-Clarke (2005), Theorem 2.8]{AC1},  they proved  a star operation version of the ``only if part'' of Proposition \ref{inter} and, moreover, in \cite[D.D. Anderson-Clarke (2005), Proposition 2.7]{AC1} and \cite[D.D. Anderson-Clarke (2006), Lemma 3.1 and Theorem 3.2]{AC2} they established several other general equivalences that, particularized in  the $v$--operation case,  are summarized in the following:

\begin{proposition} \label{bis-inter} Let $D$ be an integral domain.
\begin{itemize}
\item[\rm (a)] \hskip 10pt   $(F_1 \cap F_2 \cap ... \cap F_n)^v = F_1^v \cap F_2^v \cap ... \cap F_n^v$ for all $F_1, F_2,  ... ,  F_n \in \f(D)$  if and only if $(F:_D G)^v = (F^v :_D G^v)$ for all $F,G \in \f(D)$.
\item[\rm (b)]  \hskip 10pt The following are equivalent.
{
\begin{itemize}
\item[\rm (i)] \hskip 10pt $D$ is a $v$--domain.
\item[\rm (ii)] \hskip 10pt $D$ is integrally closed and  $(F:_D G)^v = (F^v :_D G^v)$ for all $F,G \in \f(D)$

\item[\rm (iii)] \hskip 10pt $D$ is integrally closed and $((a,b) \cap (c,d))^v = (a,b)^v \cap (c,d)^v$ for all nonzero $a ,b, c, d \in D$.
\item[\rm (iv)] \hskip 10pt $D$ is integrally closed and $((a,b) \cap (c))^v = (a,b)^v \cap (c)$ for all nonzero $a ,b, c \in D$.
\item[\rm (v)] \hskip 10pt $D$ is integrally closed and $((a,b):_D (c))^v = ((a,b)^v :_D (c))$ for all nonzero $a ,b, c \in D$.
\end{itemize}}
\end{itemize}
\end{proposition}

Note that  P$v$MD's can be characterized by   ``$t$--versions''  of the statements of  Proposition \ref{bis-inter} (b) \cite[D.D. Anderson-Clarke (2006), Theorem 3.3]{AC2}.  Moreover, in \cite[D.D. Anderson-Clarke (2006)]{AC2},   the authors also asked several questions related to distribution of the $v$--operation  over
intersections. \, One of these
questions   \cite[D.D. Anderson-Clarke (2005), Question 3.2(2)]{AC1} can be stated as:
 \sl Is it true that, if $D$ is a $v$-domain, then $(A\cap B)^{v}=A^{v}\cap B^{v}$ for all $A,B\in \F(D)$? \rm 
 
In
\cite[Mimouni (2007), Example 3.4]{M}, the author has recently answered in the negative, constructing a Pr\"ufer domain with two ideals $A, B \in \F(D)$ such
that $(A\cap B)^{v}\neq A^{v}\cap B^{v}$.  

\medskip

In a very recent  paper \cite[Anderson-Anderson-Fontana-Zafrullah (2008)]{AAFZ}, the authors  classify the integral domains that come
under the umbrella of $v$--domains, called there \emph{$\ast$--Pr\"ufer domains} for a given star operation $\ast$ (i.e., integral domains such that every nonzero finitely generated fractional ideal is $\ast$--invertible). Since $v$--Pr\"ufer domains coincide with $v$--domains,  this paper provides also direct and general proofs of several relevant quotient-based characterizations of $v$--domains given in \cite[Anderson-Mott-Zafrullah (1989), Theorem 4.1]{AMZ}.  We collect in the following theorem several of these ideal-theoretic characterizations in case of $v$--domains. For  the general statements in the star setting and for the proof  the reader can consult \cite[Anderson-Anderson-Fontana-Zafrullah (2008), Theorems 2.2 and 2.8]{AAFZ}.

\begin{theorem} \label{*prufer-2} Given an integral domain $D$,   the   following properties are  equivalent.  
\begin{itemize}
\item[\rm (i)] \hskip 15pt $D$ is  a  $v$--domain.   
\item[\rm (ii)] \hskip 15pt For  all  $A\in \F(D) $ and $F\in \f(D)$, $A\subseteq F^{v}$ implies $A^{v
}=(BF)^{v}$ for some $B\in \F(D)$.

\item[\rm (iii)] \hskip 15pt $(A:F)^{v }=(A^{v}:F) = (AF^{-1})^{v}$  for all  $A\in \F(D) $ and $F\in \f(D)$.

\item[\rm (iv)]  \hskip 15pt$(A:F^{-1})^{v }= (A^{v }:F^{-1}) = (AF)^{v }$ for all  $A\in \F(D) $ and $F\in \f(D)$.

\item[\rm (v)]  \hskip 15pt $(F:A)^{v}=(F^{v }:A)=(FA^{-1})^{v }$ for all  $A\in \F(D) $ and $F\in \f(D)$.


\item[\rm (vi)]  \hskip 15pt $(F^{v}:A^{-1})=(FA^{v})^{v }$ for all  $A\in \F(D) $ and $F\in \f(D)$.

\item[\rm (vii)] \hskip 15pt $((A+B):F)^{v}=((A:F)+(B:F))^{v }$  for all $A,B\in \F(D)$  and   $F\in \f(D)$.

 \item[\rm (viii)] \hskip 15pt $(A :(F\cap G))^v = ((A:F) + (A:G))^v$  for   all  $A\in \F(D) $ and $F, G\in \f^v(D) \ (:= \{ H \in \f(D) \mid H = H^v \})$.  

 \item[\rm (ix)]  \hskip 15pt $(((a):_{D}(b))+ ((b):_{D}(a)))^{v}=D$   for all  nonzero  $
a,b\in D$.

\item[\rm (x$_{\f}$)] \hskip 15pt $((F\cap
G)(F+G))^{v }=(FG)^{v}$ for all $F,G \in \f(D)$.

\item[\rm (x$_{\F}$)]  \hskip 15pt $((A\cap
B)(A+B))^{v }=(AB)^{v }$ for all $A,B\in \F(D)$.

\item[\rm (xi$_{\f}$)] \hskip 15pt $(F(G^{v }\cap
H^{v }))^{v }=(FG)^{v}\cap (FH)^{v}$ for all $F, G, H\in \f(D)$.

\item[\rm (xi$_{\!\f\!\F}$)] \hskip 15pt $(F(A^{v }\cap
B^{v}))^{v }=(FA)^{v }\cap (FB)^{v }$ for all $F\in \f(D)$ and $
A,B\in \F(D)$.

 \item[\rm (xii)] \hskip 15pt  If  $A, B \in \F(D)$ are $v$--invertible, then $A \cap B$ and $A+B$ are $v$--invertible.  

 \item[\rm (xiii)] \hskip 15pt If  $A, B \in \F(D)$ are $v$--invertible, then $A+B$ is $v$--invertible.  
\end{itemize}

\end{theorem}

Note that some of the previous characterizations are remarkable for various reasons.
For instance, (xiii) is interesting in that while an invertible  ideal  (respectively, $t$--invertible $t $--ideal) is finitely generated (respectively, $t$--finite) a $v$--invertible 
$v$--ideal may not be $v$--finite.  Condition (x$_{\F}$) in the star setting gives
$((A\cap B)(A+B))^{\ast }=(AB)^{\ast }$ for all $A,B\in
\F(D)$  and  for $\ast =d$ (respectively, $\ast =t$),   it is a (known) characterization of Pr\"ufer domains (respectively, P$v$MD's), but for $\ast =v$ is a brand-new characterization of $v$--domains.   
More generally, note   that, replacing in each of the statements of the previous theorem $v$ with the identity star operation $d$ (respectively, with $t$), we (re)obtain several characterizations of Pr\"ufer domains (respectively, P$v$MD's).

Franz Halter-Koch has recently shown a great deal of interest in the   paper  \cite[Anderson-Anderson-Fontana-Zafrullah (2008)]{AAFZ} and,   at the Fez Conference in June 2008, he has presented further  systematic work in the language of monoids, containing in particular  the above characterizations \cite[Halter-Koch (2009)]{HK-Fez-08}.

\medskip
{\bf Acknowledgment.} The referee's warm encouragement mixed with incisive criticism has helped improve the presentation of this article a great deal. We gratefully acknowledge the referee's contribution.

\end{section}


\printindex
\end{document}